\documentclass{amsart}
\usepackage[utf8]{inputenc}

\newtheorem{thm}{Theorem}
\newtheorem{lem}[thm]{Lemma}
\newtheorem{cor}[thm]{Corollary}
\newtheorem{rem}[thm]{Remark}

\newtheorem{defi}[thm]{Definition}

\newtheorem{prop}[thm]{Proposition}

\newtheorem*{thm*}{Theorem}
\newtheorem*{defi*}{Definition}
\newcounter{clm}
\newtheorem{claim}[clm]{Claim}

\usepackage{amsmath, amsthm, amssymb,amsfonts}
\usepackage{graphicx}
\usepackage[colorlinks=true]{hyperref}

\title{A refinement of the asymptotic expansion of Weil-Petersson volumes}
\author{Marthe Guillermit}
\date{July 2026}

\begin{document}

\maketitle

\begin{abstract}
Over the past decade, the study of the asymptotic growth of Weil-Petersson volumes of the moduli space of hyperbolic surfaces has yielded numerous results on the length spectrum and on the spectrum of the Laplacian of typical large genus surfaces. We compute the exact asymptotic value of the volume polynomials $V_{g,n}(x_1,\ldots x_n)$ for $\mathbf{x}=(x_1,\ldots x_n)$ the lengths of the boundary components such that $x_i=\mathcal{O}(\sqrt{g})$: 
$$\prod_{j=1}^{n}\frac{x_j}{2}\cdot\frac{V_{g,n}(x_1,\ldots x_n)}{V_{g,n}}\!=\!\frac{1}{2^n}\exp\left({\frac{|\mathbf{
x}|}{2}\!-\!\frac{1}{8\pi^{2}g}\left(\frac{|\mathbf{x}|}{2}\right)^{2}}\right)\!\left(1\!+\!\mathcal{O}_{n}\left(\frac{1}{\min x_j}\right)\right).$$ This result relies on the analysis of the expansion of Witten-Kontsevitch intersection numbers, for which we obtain an analogous explicit result. We also refine the bound over the coefficients of the expansion in terms of $s$ the degree of the expansion. From the expansion of the volumes, we deduce an exact estimate of the average number of non-separating simple geodesics of length of order $\sqrt{g}$. Our result therefore explains the behavior of counting functions at the cutoff $\sqrt{g}$, at which simple geodesics become negligible with respect to non-simple ones. The existence of this cut-off was conjectured by Lipnowski and Wright and proven by Wu and Xue.  
\end{abstract}
\setcounter{tocdepth}{1}
\tableofcontents

\section{Introduction}

Let $\mathcal{M}_{g,n}(\mathbf{x})$ be the moduli space of hyperbolic surfaces of genus $g$ with $n$ boundary components of lengths $\mathbf{x}=(x_1,\ldots x_n)\in \mathbb{R}_{\geq 0}^{n}$. The boundary components are either punctures for $x_i=0$ or geodesic boundaries for $x_i>0$. Let $\mathrm{Vol}_{WP}(\mathcal{M}_{g,n}(\mathbf{x}))=V_{g,n}(\mathbf{x})$ be the Weil-Petersson volume of the moduli space $\mathcal{M}_{g,n}(\mathbf{x})$. The volume of the moduli space of punctured surfaces is denoted $V_{g,n}:=V_{g,n}(0,\ldots 0)$. For fixed $g$, $V_{g,n}(\mathbf{x})$ is a polynomial in $\mathbf{x}$ whose coefficients depend on the intersection numbers of tautological classes $[\tau_{d_1} \ldots \tau_{d_n}]_{g,n}$. In high genus, the existence of an asymptotic expansion for intersection numbers $[\tau_{d_1} \ldots \tau_{d_n}]_{g,n}$ and for volumes of moduli spaces $V_{g,n}(\mathbf{x})$ at any order is one of the numerous results derived from Mirzakhani's topological recursion \cite{MZ}. Our statement precises the form of the expansion. Let us first define partial polynomials. 

\begin{defi}
Let $J\subsetneq \{1,\ldots n\}=\{\mathbf{n}\}$, $p:\mathbb{N}^n\rightarrow \mathbb{R}$ is a $J$-partial polynomial if there exist polynomials $p_1,\ldots p_m$ in $|J|$ variables and $c_1,\ldots c_m:\mathbb{N}^{n-|J|}\rightarrow\mathbb{R}$ functions of compact support such that: 
$$p(\mathbf{d})=\sum_{k=1}^mc_k(\mathbf{d}_I)\cdot p_k(\mathbf{d}_J).$$
The degree of the partial polynomial $p$ is $\mathrm{deg}(p)=\mathrm{max}_{k}(\mathrm{deg}(p_k)).$ Its support is $\mathrm{supp}(p)=\bigcup_{k=1}^m\mathrm{supp}(c_k)$.
\end{defi}

For the expansion of intersection numbers, our theorem reads: 

\begin{thm}\label{thm:internb}
Given the integers $n, s\geq 1,$, there exist:
\begin{itemize}
    \item for any $r\leq s$, $q_n^{{(r)}}$ polynomial in $\mathbf{d}=(d_1,\ldots,d_n)$ of degree $2r$ for $r\leq s$ with 
    
    \begin{equation}
        q_n^{(r)}(\mathbf{d})=\left(\frac{-1}{\pi^{2}}\right)^{r}\sum_{i_1+\ldots+i_n=2r}\frac{(2r-1)!!}{i_1!\ldots i_n!}\cdot d_1^{i_1}\ldots d_n^{i_n}+\mathcal{O}_{n,r}\left(|\mathbf{d}|^{2r-1}\right),
    \end{equation}

    \item for any $2\leq r\leq s$, $q_{n,J}^{{(r)}}$ $J$-partial polynomials with $\mathrm{deg}(q_{n,J}^{(r)})\leq 2r-4$ and $\mathrm{supp}(q_{n,J}^{(r)})\subset \{\mathbf{d}\in\mathbb{N}^{n-|J|}\lvert |\mathbf{d}|<2r-4\}$
\end{itemize}
such that for any $\mathbf{d}=(d_{1},\ldots,d_{n})$
\begin{equation}\label{myresult:ineq1}
\frac{[\tau_{d_1} \ldots \tau_{d_n}]_{g,n}}{V_{g,n}}=1+\frac{e_{n}^{(1)}(\mathbf{d})}{g}+\ldots+\frac{e_{n}^{(s)}(\mathbf{d})}{g^{s}}+\mathcal{O}_{n,s}\left(\frac{|\mathbf{d}|^{2s+2}}{g^{s+1}}\right)
\end{equation}
where $|\mathbf{d}|=\sum_{i=1}^{n}d_i$ and
\begin{align}
e_{n}^{(r)}(\mathbf{d}) &= q_n^{(r)}(\mathbf{d})+\sum_{J\subsetneq\{\mathbf{n}\}} q_{n,J}^{{(r)}}(\mathbf{d})\\
&=\left(\frac{-1}{\pi^{2}}\right)^{r}\sum_{\substack{i_1+\ldots+i_n\\=2r}}\frac{(2r-1)!!}{i_1!\ldots i_n!}\cdot d_1^{i_1}\ldots d_n^{i_n}+\mathcal{O}_{n,r}\left(|\mathbf{d}|^{2r-1}\right).
\end{align}
Moreover, there exist a constant $C_n$ depending on $n$ such that for any $r$ and $\mathbf{d}$, $e_n^{(r)}(\mathbf{d})$ admits the following bound: 
\begin{align*}
    |e_{n}^{(r)}(\mathbf{d})|\leq (C_nr)^{r}\sum_{\substack{i_1+\ldots \\+i_n\leq 2r}}\frac{d_1^{i_1}\ldots d_n^{i_n}}{i_1!\ldots i_n!}. 
\end{align*}
\end{thm}

\begin{rem}
For the dominant terms, the explicit computation immediately yields $$|[i_1,\ldots i_n]e_n^{(r)}|<\frac{(Cr)^{r}}{i_1!\ldots i_n!}$$
where $C=2/\pi^{2}$. According to the work of Hide, Macera and Thomas \cite{HMT}, the coefficients can be bounded by $(n+r)^{r}(Cr)^{cr}$ for a $c>600$. Our result improves the exponent to $(C_nr)^{r}$. One could make explicit the dependence of the constant $C$ in $n$, and we conjecture it should yield the bound $(n^{2}+1)^{r}(Cr)^{r}$. As $n$ is mostly considered as a fixed parameter in the Weil-Petersson model, we prove a weaker statement without making explicit the dependency in $n$.  
\end{rem}

From the expansion of intersection numbers one can deduce the expansion of the volumes. We give a refinement of the expansion computed by Anantharaman and Monk \cite{AM22}, see Subsection~\ref{subsec:expvol}. The degree of the polynomial in the error term is optimal.

\begin{thm}\label{thm:expvol}

 Let $n \geq 1$ be an integer.  There exists a unique family
  $(f_n^{(r)})_{r \geq 0}$ of functions and a constant $C_n$ such that
  for any integer $s \geq 0$, any genus $g \geq 1$ and any length vector $\mathbf{x} \in \mathbb{R}_{\geq 0}^n$,
  \begin{equation}
    \label{e:fn_approx}
    \left|\prod_{i=O}^n\frac{x_i}{2}\cdot\frac{V_{g,n}(\mathbf{x})}{V_{g,n}}- \sum_{r=0}^s \frac{f_{n}^{(r)}(\mathbf{x})}{g^r}\right| \leq\frac{{(C_n(s+1))^{s+1}(1+|\mathbf{x}|)}^{2s+2}}{g^{s+1}} \exp \left(\frac{|\mathbf{x}|}{2}\right).
  \end{equation}
  Furthermore, for any $r \geq 0$, the function $f_n^{(r)}$ can be expressed as
  \begin{equation}
    f_n^{(r)}(\mathbf{x})
    = \sum_{\substack{J=J_+ \sqcup J_- \subseteq\\ \{1, \ldots, n\}}}
    Q_n^{(r,J_\pm)}(\mathbf{x}) \prod_{i \in J_+} \cosh\frac{x_i}{2}
    \prod_{i \in J_-} \sinh \frac{x_i}{2},
  \end{equation}
  where $Q_n^{(r,J_\pm)}$ are uniquely defined even $n$-variable polynomial functions of degree $2r$ if $J=\{1, \ldots, n\}$ and of degree at most $2r-4$ in $\mathbf{d}_J$ if $J\subsetneq\{1, \ldots, n\}$ such that:
  \begin{align*}
      |Q_n^{(r,J_\pm)}(x_1,\ldots x_n)|\leq(C_nr)^{r}(1+|\mathbf{x}|)^{2r}. 
  \end{align*}
\end{thm}

The main polynomial $q_{n}^{(r)}$ of Theorem~\ref{thm:internb} determines $Q_n^{(r,J_{\pm})}$ for $J_{\pm}=\{1, \ldots, n\}$ and $J$-partial polynomials $q_{n,J}^{(r)}$ determine $Q_n^{(r,J_{\pm})}$ for $J_\pm\subsetneq\{1, \ldots, n\}$. 

One might hope to derive the asymptotic expansion of the volumes directly from the topological recursion, rather than through the asymptotic expansion of intersection numbers. Unfortunately, this does not seem feasible because the recursion involves poorly behaved integrals. It therefore appears necessary to first establish the asymptotic expansion of the intersection numbers.

The exact computation of the leading term of the coefficient in the expansion of intersection numbers in Theorem~\ref{thm:internb} allows us to obtain the exact value of the leading terms of the polynomials $Q_n^{(r,J_\pm)}$, and consequently to give the exact asymptotic value of the volumes up to $x_i=\mathcal{O}(\sqrt{g})$. 

\begin{thm}
\label{thm:asymsqrt}
For $\mathbf{x}(g)$ such that for every $1\leq j\leq n$, $x_j(g)\rightarrow\infty$ and $x_j=\mathcal{O}(\sqrt{g})$:
\begin{equation}
    \prod_{j=1}^{n}\frac{x_j}{2}\cdot\frac{V_{g,n}(\mathbf{x})}{V_{g,n}}=\frac{1}{2^n}\exp\left(\frac{|\mathbf{x}|}{2}-\frac{1}{8\pi^{2}g}\left(\frac{|\mathbf{x}|}{2}\right)^{2}\right)\left(1+\mathcal{O}_{n}\left(\frac{1}{\min x_i}\right)\right).
\end{equation}

For $\mathbf{x}(g)$ such that for $j\in J\subsetneq \{\mathbf{n}\}$, $x_j(g)\rightarrow\infty$ and $x_j=\mathcal{O}(\sqrt{g})$, and else $x_j(g)\rightarrow \tilde{x}_j<\infty$:
\begin{align*}
    \prod_{j\in\{\mathbf{n}\}}\frac{x_j}{2}\cdot\frac{V_{g,n}(\mathbf{x})}{V_{g,n}}&=\frac{1}{2^{|J|}}\cdot\prod_{j\notin J}\sinh\frac{\tilde{x}_j}{2}\cdot\exp\left(\frac{|\mathbf{x}_J|}{2}-\frac{|\mathbf{x}_J|^{2}}{32\pi^{2}g}\right)\\
    &\hspace{4.5cm}\cdot\left(1+\mathcal{O}_{n,|J|}\left(\frac{1}{\min_{j\in J}x_j}\right)\right).\\
\end{align*}
\end{thm}

The study of asymptotic expansion of volumes of moduli spaces has implications for our understanding of the behavior of counting functions on a typical large genus surface \cite{WX}. A counting function $N(X,L)$ counts the number of closed geodesics of length bounded by $L$ on the surface $X$. One can specify whether it counts simple closed geodesics or non-simple ones, separating or non-separating ones and study the share of simple non-separating geodesics among all closed geodesics. Lipnowski and Wright conjectured that a transition occurs at order $\sqrt{g}$. Among geodesics of length negligible compared to $\sqrt{g}$, non-separating simple geodesics are dominant. Beyond that, non-simple geodesics dominate. This conjecture was confirmed by the work of Wu and Xue \cite{WX}.  Our refinement of the asymptotic expansion of volumes enables us to observe what happens at the threshold by computing the expected value of the number of simple non-separating geodesics, where counting functions are random variables on the moduli space equipped with the normalised Weil–Petersson volume. See Subsection~\ref{subsec:counting} for details.

\begin{thm}\label{thm:count}
For $L=C\sqrt{g}$, 
\begin{align}
    \mathbb{E}_{WP}(N^{s}_{nsep}(X,L))\sim_g\frac{e^{L-\frac{L^{2}}{8\pi^{2}g}}}{2L}.
\end{align}
\end{thm}

Let us recall that the prime geodesic theorem states that for a compact surface, the number of unoriented closed geodesic $N(X,L)$ has the following asymptotic in $L$, independently of the genus $g$ and of the surface $X\in\mathcal{M}_g$: 
$$N(X,L)\sim\frac{e^L}{2L}\textit{  as  }L\to \infty.$$

Our result confirms the analogy of Lipnowski-Wright conjecture with the birthday paradox. In fact, an analogous cut-off appears when the number of people approaches $\sqrt{d}$, with $d$ the number of day in a year. If $n=o(\sqrt{d})$, then with high probability no two people in the group will have the same birthday. For there to be a non-zero probability that two people have the same birthday, $n$ must be at least of order $\sqrt{d}$. In fact, the asymptotic probability that no two people share the same birthday is $P_{n,d}=\exp \left(-\frac{n^{2}}{2d}\right)\left(1+\mathcal{O}\left(\frac{n}{d}\right)\right)$.

In Section~\ref{sec:preli}, we introduce the objects and previously known results. In Section~\ref{sec:internb}, we prove Theorem~\ref{thm:internb} on the expansion of intersection numbers. In Section~\ref{sec:expvol}, we prove Theorem~\ref{thm:expvol} and Theorem~\ref{thm:asymsqrt} on the expansion of the volumes of moduli spaces. In Section~\ref{sec:countres}, we deduce Theorem~\ref{thm:count}.

\subsection*{Acknowledgments}
We are grateful to Nalini Anantharaman for insightful suggestions, to Thomas Buc d'Alché for valuable conversations, and to Laura Monk for helpful feedbacks.

\section{Preliminaries}\label{sec:preli}
In this Section, we give the necessary background to our approach. Especially, we define the Weil-Petersson metric on the moduli spaces of hyperbolic surfaces $\mathcal{M}_{g,n}$ and we present the topological recursion formula introduced by Mirzakhani to compute the volume of moduli spaces \cite{M7}.

\subsection{Geometry of the Teichmüller and moduli spaces}

Let $S_{g,n}$ be a topological surface of genus $g$ with $n$ boundary components denoted $\left\{\beta_i\right\}_{1}^{n}$. The moduli space $\mathcal{M}_{g,n}(x_1,\ldots x_n)$ is the space of hyperbolic metrics admissible on $S_{g,n}$. Its universal cover is the Teichmüller space $\mathcal{T}_{g,n}(x_1,\ldots x_n)$ and its fundamental group the mapping class group $\mathrm{Mod}_{g,n}$. A pants decomposition $\mathcal{P}=\{\alpha_i\}_1^{3g-3+n}$ is a family of disjoint simple curves on $S_{g,n}$ such that any connected components of $S_{g,n}\setminus\mathcal{P}$ is homeomorphic to a pair of pants, a surface of genus $0$ with three boundary components. It allows us to define global coordinates on $\mathcal{T}_{g,n}(\mathbf{x})$ through the lengths of and twists along $\alpha_i$. They are called the Fenchel-Nielsen coordinates $(l_i,\tau_i)_1^{3g-3+n}$.

The Teichmüller space $\mathcal{T}_{g,n}(\mathbf{x})$ is therefore a manifold of dimension $6g-6+2n$. The action of $\mathrm{Mod}_{g,n}$ on $\mathcal{T}_{g,n}(\mathbf{x})$ is not proper, it admits fixed points corresponding to surfaces with exceptional symmetries. Therefore $\mathcal{M}_{g,n}(\mathbf{x})$ is an orbifold of real dimension $6g-6+2n$. The Teichmüller space is a symplectic manifold with symplectic form $\omega_{WP}$, the Weil-Petersson form, which has a simple expression in Fenchel-Nielsen coordinates \cite{Wol83}:

\begin{align*}
    \omega_{WP}=\sum_{i=1}^{3g-3+n}\mathrm{d}l_i\wedge\mathrm{d}t_i.
\end{align*}

It defines the Weil-Petersson volume form:
$$\mathrm{dVol}_{WP}=\frac{1}{(3g-3+n)!}\bigwedge_{1}^{3g-3+n}\omega_{WP}.$$

It is invariant under the action of the mapping class group, thus it also defines a volume form on the moduli space of surfaces. The Weil-Petersson volume of the moduli space is finite \cite{Bers}.

\subsection{Volume polynomials and intersection numbers}

In \cite{M7b}, Mirzakhani proved that Weil-Petersson volumes of the moduli space of bordered surfaces $V_{g,n}(\mathbf{x})$ are polynomials in the lengths of the boundary components $\mathbf{x}$ whose coefficients linearly depend on Witten-Kontsevitch intersection numbers \cite{M7}. 

\begin{defi}
The normalized intersection number associated to $\mathbf{d}=(d_1,\ldots d_n)$ is: 
$$[\tau_{d_1}\ldots\tau_{d_n}]_{g,n}:=\frac{\prod_{i=1}^{n}2^{2d_i}(2d_i+1)!!}{d_0!}\int_{\overline{\mathcal{M}}_{g,n}}\psi_1^{d_1}\ldots \psi_n^{d_n}\cdot \omega_{WP}^{d_0}$$
where $\psi_i$ is the first Chern class of the $i$-th tautological line bundle on the Deligne-Mumford compactification of the moduli space $\overline{\mathcal{M}}_{g,n}$ and $d_0=3g-3+n-|\mathbf{d}|$.
\end{defi}

The understanding of the algebraic construction of intersection numbers will not be needed in our proof. The only properties one should keep in mind are the following: 

\begin{prop}\label{prop:internbMZ}
\begin{enumerate}
    \item Intersection numbers are symmetric: for any $\sigma\in\mathfrak{S}_n$, 
    $$[\tau_{d_{\sigma(1)}}\ldots\tau_{d_{\sigma(n)}}]_{g,n}=[\tau_{d_1}\ldots\tau_{d_n}]_{g,n}.$$
    \item Intersection numbers lie in $\mathbb{Q}[\pi^{2}]$. More precisely, for any $\mathbf{d}$:
    $$[\tau_{d_{1}}\ldots\tau_{d_{n}}]_{g,n}\in\pi^{6g-6+2n-2|\mathbf{d}|}\cdot\mathbb{Q}.$$
    \item Intersection numbers vanish for $|\mathbf{d}|>3g-3+n$. 
\end{enumerate}
\end{prop}

\begin{proof}[Sketch of proof]
We refer to \cite{HM} for the construction of tautological classes and of Deligne-Mumford compactification. Let us however give a few arguments of why these properties hold. The symmetry is immediate. For the two latter properties, let us first recall that Witten-Kontsevitch intersection numbers were originally defined for $|\mathbf{d}|=3g-3+n$ \cite{Wit}. As the first Chern classes $\psi_i$ live in $H^{2}(\overline{\mathcal{M}}_{g,n},\mathbb{Z})$ and $\overline{\mathcal{M}}_{g,n}$ has real dimension $6g-6+2n$, the following expression makes sense: 
$$[\tau_{d_1}\ldots\tau_{d_n}]_{g,n}:=\int_{\overline{\mathcal{M}}_{g,n}}\psi_1^{d_1}\ldots \psi_n^{d_n}.$$ 
The integral lies in $\mathbb{Q}$ because $\overline{\mathcal{M}}_{g,n}$ is an orbifold. For $|\mathbf{d}|<3g-3+n$, it is natural to multiply by the tautological class $\kappa_1\in H^{2}(\overline{\mathcal{M}}_{g,n},\mathbb{Z})$, which verifies $[\omega_{WP}]=2\pi^{2}\kappa_1$, to obtain a top degree class. Thus, as we defined it, $[\tau_{d_{1}}\ldots\tau_{d_{n}}]_{g,n}$ lies in $\pi^{2(3g-3+n-|\mathbf{d}|)}\cdot \mathbb{Q}$. This yields the second property. Finally, for $|\mathbf{d}|>3g-3+n$, $$\psi_1^{d_1}\ldots \psi_n^{d_n}\in H^{2|\mathbf{d}|}(\overline{\mathcal{M}}_{g,n},\mathbb{Z})$$ which vanishes as $2|\mathbf{d}|$ is greater than the dimension of the moduli space, and we obtain the third property. 
\end{proof}

Having defined intersection numbers, we now express the Weil–Petersson volume polynomials in terms of them \cite{M7}.

\begin{thm}[Mirzakhani]
The coefficients of the volume polynomial 
$$V_{g,n}(L_1,\ldots L_n)=\sum_{|\mathbf{d}|\leq 3g-3+n}C_g(\mathbf{d})\cdot L_1^{2d_1}\ldots L_{n}^{2d_n}$$
are expressed as 
$$C_g(d_1,\ldots d_n)=\frac{1}{2^{2|\mathbf{d}|}\mathbf{d}!}[\tau_{d_1}\ldots\tau_{d_n}]_{g,n}.$$
Here $\mathbf{d}=(d_1\ldots d_n)$, $\mathbf{d}!=\prod_{i=0}^{n}d_i!$ and $|\mathbf{d}|=\sum_{i=0}^{n}d_i$. 
\end{thm}

\subsection{Topological recursion}
From the topological recursion on volumes, Mirzakhani deduces a recursive formula for intersection numbers \cite{M7}. Mirzakhani's recursion formula expresses $V_{g,n}$ in terms of $V_{g',n'}$ for $S_{g',n'}$ a connected component of $S_{g,n}\setminus P$, where $P$ runs over all the pair of pants embedded in $S_{g,n}$ such that $\beta_1\subset \partial P$. The contribution of $V_{g',n'}$ to $V_{g,n}$ depends on the topological type of $P$. If $P$ disconnects the surface, the corresponding $V_{g',n'}$'s contribute to the third term of the recursion. If $\beta_j\subset\partial P$ for $j\neq 1$, ie $P$ shares an other boundary component with $S_{g,n}$, it contributes to the first term. Else, it contributes to the second term. Figure~\ref{fig:recursion} illustrates the three different cases. In all three cases, the opposite of the Euler characteristic decreases between $S_{g,n}$ and $S_{g',n'}$. The identity of polynomials implies an identity on the coefficients, which is called the topological recursion for intersection numbers. It relates intersection numbers on the moduli space $\mathcal{M}_{g,n}$ to those on moduli spaces $\mathcal{M}_{g',n'}$ for $2g'+n'<2g+n$ \cite{M7}. 

\begin{figure}
    \centering
    \includegraphics[width=\textwidth]{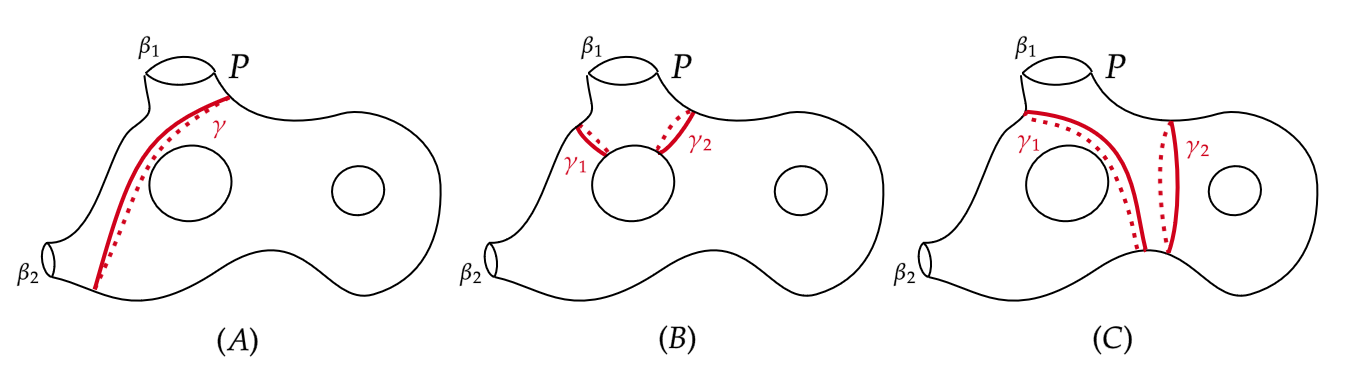}
    \caption{The three terms of the topological recursion.}
    \label{fig:recursion}
\end{figure}

\begin{thm}[Mirzakhani]
Put $a_{i} =(1-2^{1-2i})\,\zeta(2i),$ where $\zeta$ is the Riemann zeta function and $i\in\mathbb{Z}_{\geq 0}$.
Then
$$
[ \tau_{d_{1}}\ldots\tau_{d_{n}}]_{g,n}= \mathcal{A}_{g,n}(\mathbf{d}) + \mathcal{B}_{g,n}(\mathbf{d})+ \mathcal{C}_{g,n}(\mathbf{d}),
$$
where 
\begin{align*}
\mathcal{A}_{g,n}(\mathbf{d})=&\;8 \; \sum_{j=2}^{n} \sum_{l=0}^{d_{0}} (2d_{j}+1) \; a_{l} \left[\tau_{d_{1}+d_{j}+l-1} \prod_{i\not=1,j} \tau_{d_{i}}\right]_{g,n-1},\\
\mathcal{B}_{g,n}(\mathbf{d})= &\;16 \;\sum_{l=0}^{d_{0}} \sum_{\genfrac{}{}{0pt}{}{k_{1}+k_{2}=}{=l+d_{1}-2}} a_{l} \left[\tau_{k_{1}} \tau_{k_{2}} \prod_{i\not=1} \tau_{d_{i}}\right]_{g-1,n+1},\\
\mathcal{C}_{g,n}(\mathbf{d})=&\;16\sum_{\genfrac{}{}{0pt}{}{g_{1}+g_{2}=g}{I\amalg J=\{\mathbf{n}\}}} \sum_{l=0}^{d_{0}} \sum_{\genfrac{}{}{0pt}{}{k_{1}+k_{2}=}{=l+d_{1}-2}} a_{l} \; \left[\tau_{k_{1}} \prod_{i\in I } \tau_{d_{i}}\right]_{g_1,|I|+1} \cdot \left[\tau_{k_{2}} \prod_{i\in J} \tau_{d_{i}}\right]_{g_2,|J|+1}
\end{align*}   
where $d_0=3g-3+n-|\mathbf{d}|$.
\end{thm}

The sequence $(a_i)_{i\geq 0}$ has the following properties \cite{MZ}: 

\begin{lem}[Mirzakhani-Zograf]\label{lem:sumoveri}
For $j \in {\mathbb Z}$,  $$\sum_{i=0}^{\infty} i^{j} (a_{i+1}-a_{i})<\infty.$$ 
In particular, 
$$\sum_{i=0}^{\infty} (a_{i+1}-a_{i})=\frac{1}{2}\qquad\text{ and }\qquad \sum_{i=0}^{\infty} i(a_{i+1}-a_{i})=\frac{1}{4}$$
\end{lem}

Lemma~\ref{lem:sumoveri} implies the following corollary:

\begin{cor}\label{cor:polandai}
For $p$ a polynomial in $n+1$ variables of degree $m$,
$$\sum_{i=0}^{\infty} p(i,d_1,d_2,\ldots d_n)(a_{i+1}-a_{i})=\tilde{p}(d_1,d_2,\ldots d_n)$$ with $\tilde{p}$ a polynomial in $n$ variables of degree $m$. The coefficient before $d_1^{i_1}\ldots d_n^{i_n}$ for $\sum_{j=0}^{n}i_j=m$ in $\tilde{p}$ is half the coefficient of $d_1^{i_1}\ldots d_n^{i_n}$ in $p$.
\end{cor}

\subsection{Ratios of Weil-Petersson volumes}
From topological recursion and other relations on intersection numbers, Mirzakhani and Zograf deduce the expansion of certain ratios of moduli spaces \cite{MZ}. The ratios concern volumes of moduli spaces of punctured surfaces, ie $V_{g,n}=V_{g,n}(0,\ldots 0)$. It can then be generalized through the study of the ratio $V_{g,n}(\mathbf{x})/V_{g,n}$, whose expansion is described in our Theorem~\ref{thm:expvol}. 

\begin{thm}[Mirzakhani-Zograf]\label{thm:expratios}
\begin{enumerate}
\item Given $n\geq 0$, $i\geq 1$, there exist $b_{n}^{(i)},\;c_{n}^{(i)}$ independent of $g$ such that for any $s\geq1$, as $g\rightarrow \infty$:
\begin{align}
\frac{4\pi^2 (2g-2+n)V_{g,n}}{V_{g,n+1}} = &1+ \frac{b_{n}^{(1)}}{g}+\ldots+\frac{b_{n}^{(s)}}{g^{s}}
+\mathcal{O}_{n,s}\left(\frac{1}{g^{s+1}}\right),\label{equ:n/n+1}\\
\frac{V_{g,n}} {V_{g-1,n+2}}=&1+ \frac{c_{n}^{(1)}}{g}+\ldots+\frac{c_{n}^{(s)}}{g^{s}}+\mathcal{O}_{n,s}\left(\frac{1}{g^{s+1}}\right).\label{equ:g-1/g}
\end{align}
\end{enumerate}
\end{thm}

The above theorem allows us to compare the volumes of moduli spaces of punctured hyperbolic surfaces. The following estimate is a direct corollary.

\begin{cor}\label{cor:ratio} For $S_{g,n}$, $S_{g',n'}$ two hyperbolic surfaces, the ratio of the volumes of their moduli spaces admits the following asymptotic expression:
\begin{equation}\label{equ:ratio}
    \frac{V_{g',n'}}{V_{g,n}}=\frac{1}{(8\pi^{2}g)^{\chi(S_{g',n'})-\chi(S_{g,n})}}+\mathcal{O}_{n,n'}\left(\frac{1}{g^{\chi(S_{g',n'})-\chi(S_{g,n})+1}}\right).
\end{equation}
\end{cor}

The moduli spaces of two surfaces $S_{g,n}$ and $S_{g',n'}$ of equal Euler characteristic have asymptotically the same size. In high genus, Weil-Petersson volumes only depend on the Euler characteristic. Suppose $2g+n>2g'+n'$, thus $\chi(S_{g,n})<\chi(S_{g',n'})$ and the ratio tends to zero. In topological recursion, sub-surfaces $S_{g',n'}$ appears as we take off a pair of pants of $S_{g,n}$. Thus we are also interested by the ratio $$\frac{V_{g_1,n_1}\cdot V_{g_2,n_2}}{V_{g,n}},$$
which correspond to the third case of topological recursion, where the surface $S_{g,n}\setminus P$ is disconnected.

\begin{lem}[Mirzakhani-Zograf] \label{lem:sumratioMZ}
Fix $n_{1}, n_{2},s \geq 0.$ Then
    \begin{equation*}
        \sum_{\genfrac{}{}{0pt}{}{g_{1}+g_{2}=g}{2g_{i}+n_{i} \geq s,\; i=1,2}} \frac{V_{g_1,n_1} \cdot V_{g_{2},n_{2}}}{V_{g,n_1+n_2}}= \mathcal{O}_{n,s}\left(\frac{1}{g^{s}}\right).
    \end{equation*}
\end{lem}

Here one should think of $S_{g_1,n_1}$ and $S_{g_2,n_2}$ as the surfaces obtained by cutting $S_{g,n}$ along a closed non-separating geodesic and sum the ratios over the possibilities. Asymptotically, up to order $\mathcal{O}_{n,s}\left(\frac{1}{g^{s}}\right)$, one only sees the cases where the Euler characteristic of one of the separated surfaces is strictly bounded by $s$. 

\subsection{Previously known result on the expansion of intersection numbers}\label{subsec:expinternb}
Mirzakhani and Zograf established the existence of the asymptotic expansions for intersection numbers \cite{MZ}. 

\begin{thm}[Mirzakhani-Zograf] 
Given the integers $n, s\geq 1,$ and ${d}=(d_{1},\ldots, d_{n})$, there exist polynomials $Q_{n}^{{(s+1)}}(d_1,\ldots,d_n)$ of degree $2(s+1)$
and $q_{n}^{{(r)}}(d_1,\ldots,d_n)$ of degrees $2r$ for $r\leq s$ such that as $g\rightarrow\infty$, for any $d=(d_{1},\ldots,d_{n})$:
\begin{equation*}
\left|\frac{[\tau_{d_1} \ldots \tau_{d_n}]_{g,n}}{V_{g,n}}-1-\frac{e_{n}^{(1)}(\mathbf{d})}{g}-\ldots-\frac{e_{n}^{(s)}(\mathbf{d})}{g^{s}}\right| 
\leq \frac{Q_{n}^{(s+1)}(d_{1},\ldots,d_{n})}{g^{s+1}},
\end{equation*}
and
\begin{equation*}\label{thm:ineq2}
|e_{r}^{(s)}(\mathbf{d})| \leq q_{n}^{r}(d_{1},\ldots,d_{n}).
\end{equation*}
\end{thm}

\subsection{Previously known result on the expansion of the volumes}\label{subsec:expvol}
Anantharaman and Monk proved the expansion of the volumes has the following form \cite{AM22}: 

\begin{thm}[Anantharaman-Monk]

Let $n \geq 1$ be an integer.  There exists a unique family
$(f_n^{(r)})_{r \geq 0}$ of functions such that
  for any integer $s \geq 0$, any genus $g \geq 1$ and any length
  vector $\mathbf{x} \in \mathbb{R}_{\geq 0}^n$,
  \begin{equation}
    \frac{V_{g,n}(\mathbf{x})}{V_{g,n}}
    = \sum_{r=0}^s \frac{f_{n}^{(r)}(\mathbf{x})}{g^r}
    + \mathcal{O}_{n,s}\left(\frac{{|\mathbf{x}|}^{3s+1}}{g^{s+1}} \exp \frac{|\mathbf{x}|}{2}\right).
  \end{equation}
  Furthermore, for any $r \geq 0$, the function $f_n^{(r)}$ can be expressed as
  \begin{equation}
    \label{eq:fn_sinh}
    f_n^{(r)}(\mathbf{x})
    = \sum_{J_+ \sqcup J_- \subseteq \{1, \ldots, n\}}
    Q_n^{(r,J_\pm)}(\mathbf{x}) \prod_{i \in J_+} \cosh\frac{x_i}{2}
    \prod_{i \in J_-} \frac{2}{x_i}\sinh \frac{x_i}{2},
  \end{equation}
  where $Q_n^{(r,J_\pm)}$ are uniquely defined even $n$-variable polynomial functions.
  
  Furthermore, there exists constants $D_{n,r}, A_r \geq 0$ such that the polynomial function $Q_{n}^{(r,J_\pm)}$ can
  be expressed as a polynomial of degree $\leq D_{n,r}$, and its coefficients can be written as linear combinations
  (independent of $g$) of the $[\tau_{d_1}\ldots\tau_{d_n}]/V_{g,n}$ for $\mathbf{d}$ such that
  $|{\mathbf{d}}|_\infty \leq A_r$.

\end{thm}

Our approach precises the value of $D_{n,r}$ and $A_r$. Moreover, it yields the optimal degree for the polynomial in the error term, $2s+2$ instead of $3s+1$.

\subsection{Counting simple closed geodesics}\label{subsec:counting}

Wu and Xue proved in \cite{WX} Lipnowski-Wright conjecture. At the length scale $\sqrt{g}$, one observes a change in the behavior of geodesics. Among geodesics whose length is negligible compared to $\sqrt{g}$, simple non-separating ones dominate, whereas non-simple geodesics become prevalent once a larger window is considered. Their result holds with high probability, where the probability measure is the renormalized Weil-Petersson volume, which we denote $\mathbb{P}_{WP}$. Wu and Xue's result reads: 

\begin{thm}[Wu-Xue]\label{thm:WX}
The following two probabilities hold:
\begin{enumerate}
\item if $L(g)$ satisfies that for some fixed $\epsilon_0>0$, $$L(g)\geq (1-\epsilon_0)\ln g \text{ and }\lim \limits_{g\to \infty}\frac{L^2(g)}{g}=0,$$ then there exists a function $\delta(g)>0$ satisfying $\lim \limits_{g\to \infty}\delta(g)=0$ such that
$$\lim \limits_{g\to \infty}\mathbb{P}_{WP}\left(X_g\in\mathcal{M}_g;\ \left|1-\frac{N_{nsep}^s(X_g,L(g))}{N(X_g,L(g))} \right|<\delta(g)\right)=1.$$

\item if $L(g)$ satisfies $$\lim \limits_{g\to \infty}\frac{L^2(g)}{g}=\infty,$$ then
$$\lim \limits_{g\to \infty}\mathbb{P}_{WP}\left(X_g\in\mathcal{M}_g; \left|1-\frac{N^{ns}(X_g,L(g))}{N(X_g,L(g))} \right|<\frac{g}{L(g)^2}\right)=1.$$
\end{enumerate}
where
\begin{itemize}
    \item $N(X,L):=\#\left\{\gamma\text{ closed geodesic on }X\text{ such that }l_X(\gamma)\leq L\right\}$
    \item $N^{ns}(X,L):=\#\left\{\gamma\text{ closed non-simple geodesic on }X\text{ such that }l_X(\gamma)\leq L\right\}$
    \item $N^s_{nsep}(X,L):=\#\left\{\begin{array}{c}
          \gamma\text{ closed simple non-separating}\\
          \text{geodesic on }X\text{ such that } l_X(\gamma)\leq L
    \end{array}\right\}$.
\end{itemize}
\end{thm}

To obtain a result over the ratio of counting functions, the first step is to compare their expected values. The asymptotic value of the number of closed geodesic is deterministic. By the prime geodesic theorem of Huber \cite{huber}, for any $X\in\mathcal{M}_{g,n}$:
$$N(X,L)\sim\frac{e^L}{2L}.$$

The work of Mirzakhani and Petri essentially shows that up to a certain $L$, the expected number of simple non-separating geodesics on a typical surface has the same asymptotic \cite{MP19}. It is computed through Mirzakhani's integration formula for simple non-separating geodesics \cite{M7b}: 

\begin{thm}[Mirzakhani's Integration Formula]\label{thm:integr}
Let $\gamma\subset S_{g,n}$ be a simple non-separating closed geodesic, $\mathcal{O}_\gamma$ its orbit for the action of the mapping class group. For $F:\mathbb{R}_{\geq 0}\rightarrow\mathbb{R}$, one defines 
$$\begin{array}{cccc}
     F^{\gamma}:&\mathcal{M}_{g,n}&\rightarrow&\mathbb{R} \\
     & X&\mapsto&\displaystyle\sum_{\alpha\in\mathcal{O}^\gamma}F(l_\alpha(X)).
\end{array}$$
Then the integral of $F^\gamma$ over $\mathcal{M}_{g}$ with respect to Weil-Petersson metric is given by: 
$$\int_{\mathcal{M}_{g}}F^\gamma\mathrm{dVol}_{WP}(X)=\frac{1}{2}\int_{\mathbb{R}_{\geq 0}}F(x)V_{g-1,n+2}(x,x)x\mathrm{d}x.$$
\end{thm}

Thanks to Markov inequality, the estimates of the expected values yield the first result, ie the case $L=o(\sqrt{g})$. For the second part of the theorem, their proof relies on new estimates over intersection numbers.  

\section{Expansion of intersection numbers}\label{sec:internb}
\subsection{Outline of the proof}

Let us recall the statement of Theorem~\ref{thm:internb}: 
\begin{thm*}
Given the integers $n, s\geq 1,$, there exist:
\begin{itemize}
    \item for any $r\leq s$, $q_n^{{(r)}}$ polynomial in $\mathbf{d}=(d_1,\ldots,d_n)$ of degree $2r$ for $r\leq s$ with 
    
    \begin{equation}
        q_n^{(r)}(\mathbf{d})=\left(\frac{-1}{\pi^{2}}\right)^{r}\sum_{i_1+\ldots+i_n=2r}\frac{(2r-1)!!}{i_1!\ldots i_n!}\cdot d_1^{i_1}\ldots d_n^{i_n}+\mathcal{O}_{n,r}\left(|\mathbf{d}|^{2r-1}\right),
    \end{equation}

    \item for any $2\leq r\leq s$, $q_{n,J}^{{(r)}}$ $J$-partial polynomials with $\mathrm{deg}(q_{n,J}^{(r)})\leq 2r-4$ and $\mathrm{supp}(q_{n,J}^{(r)})\subset \{\mathbf{d}\in\mathbb{N}^{n-|J|}\lvert |\mathbf{d}|<2r-4\}$
\end{itemize}
such that for any $\mathbf{d}=(d_{1},\ldots,d_{n})$
\begin{equation}
\frac{[\tau_{d_1} \ldots \tau_{d_n}]_{g,n}}{V_{g,n}}=1+\frac{e_{n}^{(1)}(\mathbf{d})}{g}+\ldots+\frac{e_{n}^{(s)}(\mathbf{d})}{g^{s}}+\mathcal{O}_{n,s}\left(\frac{|\mathbf{d}|^{2s+2}}{g^{s+1}}\right)
\end{equation}
where $|\mathbf{d}|=\sum_{i=1}^{n}d_i$ and
\begin{align}
e_{n}^{(r)}(\mathbf{d}) &= q_n^{(r)}(\mathbf{d})+\sum_{J\subsetneq\{\mathbf{n}\}} q_{n,J}^{{(r)}}(\mathbf{d})\\
&=\left(\frac{-1}{\pi^{2}}\right)^{r}\sum_{\substack{i_1+\ldots+i_n\\=2r}}\frac{(2r-1)!!}{i_1!\ldots i_n!}\cdot d_1^{i_1}\ldots d_n^{i_n}+\mathcal{O}_{n,r}\left(|\mathbf{d}|^{2r-1}\right).
\end{align}
Moreover, there exist a constant $C_n$ depending on $n$ such that for any $r$ and $\mathbf{d}$, $e_n^{(r)}(\mathbf{d})$ admits the following bound: 
\begin{align*}
    |e_{n}^{(r)}(\mathbf{d})|\leq (C_nr)^{r}\sum_{\substack{i_1+\ldots \\+i_n\leq 2r}}\frac{d_1^{i_1}\ldots d_n^{i_n}}{i_1!\ldots i_n!}. 
\end{align*}
\end{thm*}

We mostly follow the proof of Mirzakhani and Zograf, except for the contribution of the third term of the recursion formula. Through the proof, we are able to compute the dominant part of the coefficients of the expansion $e_n^{(s)}$. We emphasize that it would be much more difficult to write down explicitly the relations for non dominant terms and therefore to identify them. In fact, the $J$-partial polynomials are of degree strictly less than the main polynomial, so we can ignore them, and we don't need to consider the third term of the recursion to compute dominant coefficients. We prove the bound $(C_ns)^{s}$ on the coefficients of the polynomials separately in Subsection~\ref{subsec:bound}.

Before going into details, we make a few elementary remarks on the behaviour of sums of polynomials. 

\begin{rem}\label{rem:convo}

\begin{enumerate}
    \item $\mathrm{Faulhaber's \;formula.}$ Let $r\in\mathbb{N}$, 
$$\sum_{i=1}^{d}i^{r}=\sum_{k=0}^{r+1}\frac{\mathrm{Ber}_{r+1-k}}{r+1}d^{k}$$ is a polynomial in $d$ of degree $r+1$ with $\mathrm{Ber}_{k}$ the $k$-th Bernoulli number and $\mathrm{Ber}_0=1$. 
    \item Let $p_1$ be a polynomial of degree $r_1$ and $p_2$ be a polynomial of degree $r_2$. Then the convolution of $p_1$ and $p_2$ 
    $$p_1\ast p_2(k)=\sum_{k_1+k_2=k}p_1(k_1)\cdot p_2(k_2)$$
    is a polynomial of degree $r_1+r_2+1$.
    \item Let $p$ be a polynomial in two variables $(d_1,d_2)$ of degree $r$. Then
    $$q(d)=\sum_{d_1+d_2=d}p(d_1,d_2)$$
    is a polynomial in $d$ of degree $r+1$.
\end{enumerate}
\end{rem}

Partial polynomials are also well-behaved in the following sense. 

\begin{rem}\label{rem:partpol}\begin{enumerate}
    \item The product of a $J$-partial polynomial with a polynomial is a $J$-partial polynomial.
    \item The sum of two $J$-partial polynomials is a $J$-partial polynomial.
    \item Let $p$ be a $J$-partial polynomial in $n$ variables of degree $r$. Then 
    $$q(d_1,\ldots d_n)=\sum_{i=0}^{d_1}p(i,d_2\ldots d_n)$$
    is a $J$-partial polynomial of degree $r+1$ if $1\in J $ and of degree $r$ if $1\notin J$. In both cases the support remains unchanged. 
\end{enumerate}
\end{rem}

We are now ready to prove Theorem~\ref{thm:internb}.

\begin{proof}[Proof of Theorem~\ref{thm:internb}] The base case of the induction on $s$ can be found in \cite{AM22}: 
$$\frac{[\tau_{d_1} \ldots \tau_{d_n}]_{g,n}}{V_{g,n}}=1+\mathcal{O}_{n}\left(\frac{|\mathbf{d}|^{2}}{g}\right).$$

For the induction step, we suppose that for any $n$, intersection numbers admit a $s$-th order asymptotic expansion: 
$$\frac{[\tau_{d_1} \ldots \tau_{d_n}]_{g,n}}{V_{g,n}}=1+\frac{e_{n}^{(1)}(\mathbf{d})}{g}+\ldots+\frac{e_{n}^{(s)}(\mathbf{d})}{g^{s}}+\mathcal{O}_{n,s}\left({\frac{|\mathbf{d}|^{2s+2}}{g^{s+1}}}\right)$$
such that for any $r\leq s$, $e_n^{(r)}$ is as described in Theorem~\ref{thm:internb}. We aim to prove there exist $e_n^{(s+1)}$, $q_n^{(s+1)}$, $q_{n,J}^{(s+1)}$ as in Theorem~\ref{thm:internb} such that:  

\begin{align*}
\frac{[\tau_{d_1} \ldots \tau_{d_n}]_{g,n}}{V_{g,n}}=1+\frac{e_{n}^{(1)}(\mathbf{d})}{g}+\ldots+\frac{e_{n}^{(s+1)}(\mathbf{d})}{g^{s+1}}+\mathcal{O}_{n,s}\left(\frac{|\mathbf{d}|^{2(s+2)}}{g^{s+2}}\right)
\end{align*}

and 
\begin{align*}
   e_n^{(s+1)}(\mathbf{d})&=q_n^{(s+1)}(\mathbf{d})+\sum_{J\subsetneq\{\mathbf{n}\}}q_{n,J}^{(s+1)}(\mathbf{d})\\
   &=\left(\frac{-1}{\pi^{2}}\right)^{s+1}\sum_{\substack{i_1+\ldots+i_n\\=2s+2}}\frac{(2s+1)!!}{i_1!\ldots i_n!}\cdot d_1^{i_1}\ldots d_n^{i_n}+\mathcal{O}_{n,s}\left(|\mathbf{d}|^{2s+1}\right)
\end{align*}

\subsection{Asymptotic expansion of the discrete derivative}
As in \cite{MZ}, we first study the discrete derivative of intersection numbers and decompose it in three terms corresponding to the three terms of topological recursion. 

\begin{align}\label{eq:discrderivative}
    \frac{\delta_1[\tau_{d_{1}}\,\ldots \tau_{d_n}]_{g,n}}{V_{g,n}}&=\frac{[\tau_{d_{1}+1}\,\ldots \tau_{d_n}]_{g,n}-[\tau_{d_{1}}\,\ldots \tau_{d_n}]_{g,n}}{V_{g,n}}\\
    &=\frac{{\delta_1\mathcal{A}}_{g,n}(\mathbf{d})}{V_{g,n}}+\frac{{\delta_1\mathcal{B}}_{g,n}(\mathbf{d})}{V_{g,n}}+\frac{{\delta_1\mathcal{C}}_{g,n}(\mathbf{d})}{V_{g,n}}
\end{align}

For each term we show that the induction hypothesis implies it admits an expansion up to order $(s+1)$, that the $(s+1)$-th coefficient of this expansion is the sum of a polynomial and of $J$-partial polynomials and that the error term is $\mathcal{O}_{n,s}\left(|\mathbf{d}|^{2s+4}/g^{s+2}\right)$.  

\subsection{Contribution of the first term $\frac{\delta_1{\mathcal{A}}_{g,n}(\mathbf{d})}{V_{g,n}} $. }

The first derivative of the first term is \cite{MZ}:
    \begin{align*}
        \frac{\delta_1{\mathcal{A}}_{g,n}(\mathbf{d})}{V_{g,n}} &= \frac{1}{4 \pi^2 (2g-3+n)}\cdot \frac{4\pi^2 (2g-3+n)V_{g,n-1}}{V_{g,n}}\cdot\frac{\delta_1{\mathcal{A}}_{g,n}(\mathbf{d})}{V_{g,n-1}},\\
        \frac{\delta_1{\mathcal{A}}_{g,n}(\mathbf{d})}{V_{g,n-1}} &= 8\;\sum_{j=2}^{n}\sum_{i=0}^{d_0} (a_{i-1}-a_{i})(2d_j+1) \frac{[\tau_{d_{1}+d_{j}+i-1}\tau_{d_{2}}\ldots\widehat{\tau_{d_{j}}}\ldots \tau_{d_{n}}]_{g,n-1}}{V_{g,n-1}}
    \end{align*}
where, by convention, $a_{-1}=0$.
We consider separately the expansion of the factor independent of $\mathbf{d}$: 
$$\frac{1}{4 \pi^2 (2g-3+n)}\cdot \frac{4\pi^2 (2g-3+n)V_{g,n-1}}{V_{g,n}}$$
and the expansion of the term that depends on $\mathbf{d}$: 
$$\frac{\delta_1{\mathcal{A}}_{g,n}(\mathbf{d})}{V_{g,n-1}}.$$ 
On one side, we have a product of asymptotic expansions. 
\begin{rem}\label{rem:prodexp}
Let $F, H:\mathbb{N}\rightarrow\mathbb{R}$ admit $s$-th order asymptotic expansions 
\begin{align*}
F(g)=e_F^{(0)}+\frac{e_F^{(1)}}{g}+\ldots + \frac{e_F^{(s)}}{g^{s}}+\mathcal{O}_{s}\left(\frac{1}{g^{s+1}}\right)
\end{align*}
and 
\begin{align*}
    H(g)=\frac{e_H^{(k)}}{g^k}+\ldots + \frac{e_H^{(s)}}{g^{s}}+\mathcal{O}_{s}\left(\frac{1}{g^{s+1}}\right)
\end{align*}
then $F\cdot H$ admits a $s+k$-th order expansion: 
\begin{align*}
    F\cdot H(g)=\frac{e_{F\cdot H}^{(k)}}{g^k}+\ldots + \frac{e_{F\cdot H}^{(s+k)}}{g^{s+k}}+\mathcal{O}_{s}\left(\frac{1}{g^{s+k+1}}\right).
\end{align*}
\end{rem}
According to Remark~\ref{rem:prodexp}, $$\frac{1}{4 \pi^2 (2g-3+n)}\cdot \frac{4\pi^2 (2g-3+n)V_{g,n-1}}{V_{g,n}}$$ admits an expansion up to order $s+1$ because $$\frac{1}{4 \pi^2 (2g-3+n)}=\mathcal{O}_{n,s}\left({\frac{1}{g}}\right).$$
We write the expansion of the factor: 
$$\frac{1}{4 \pi^2 (2g-3+n)}\cdot \frac{4\pi^2 (2g-3+n)V_{g,n-1}}{V_{g,n}}=\frac{\mathfrak{a}_n^{(1)}}{g}+\ldots + \frac{\mathfrak{a}_n^{(s+1)}}{g^{s+1}}+\mathcal{O}_{n,s}\left(\frac{1}{g^{s+2}}\right)$$
where $$\mathfrak{a}_n^{(r)}=\frac{1}{8\pi^{2}}\sum_{r_1+r_2=r}\left(\frac{n-3}{2}\right)^{r_1}\cdot b_{n-1}^{(r_2)}.$$ 
In particular, $\mathfrak{a}_n^{(0)}=\frac{1}{8\pi^{2}}$.
On the other side, $\delta_1{\mathcal{A}}_{g,n}(\mathbf{d})/V_{g,n-1}$ also admits a $s$-th order expansion. This follows directly from the expansion of $\frac{[\tau_{d'_{1}}\ldots\tau_{d'_{n-1}}]_{g,n-1}}{V_{g,n-1}}$: 
\begin{align*}
    \frac{\delta_1{\mathcal{A}}_{g,n}(\mathbf{d})}{V_{g-1,n+1}}&=8\;\sum_{j=2}^{n}\sum_{i=0}^{d_0} (a_{i-1}-a_{i})(2d_j+1)\cdot e_{n-1}^{(0)}(d_1+d_j+i-1,d_2\ldots\widehat{d_j},\ldots d_n)\\
    &+\frac{1}{g}\cdot 8\;\sum_{j=2}^{n}\sum_{i=0}^{d_0} (a_{i-1}-a_{i})(2d_j+1)\cdot e_{n-1}^{(1)}(d_1\!+\!d_j\!+\!i\!-\!1,d_2\ldots\widehat{d_j},\ldots d_n)\\
    &+\ldots \\
    &+\frac{1}{g^{s}}\cdot 8\;\sum_{j=2}^{n}\sum_{i=0}^{d_0} (\!a_{i-1}\!-\!a_{i}\!)(2d_j+1) \cdot e_{n-1}^{(s)}(d_1\!+\!d_j\!+\!i\!-\!1,d_2\ldots\widehat{d_j},\ldots d_n)\\
    &+8\;\sum_{j=2}^{n}\sum_{i=0}^{d_0} (a_{i-1}-a_{i})(2d_j+1)\cdot\mathcal{O}_{n,s}\left(\frac{|\mathbf{d}|^{2s+2}}{g^{s+1}}\right)\\
    &=e_{\delta_1\mathcal{A}_n/V_{n-1}}^{(0)}(\mathbf{d})\!+\!\frac{e_{\delta_1\mathcal{A}_n/V_{n-1}}^{(1)}(\mathbf{d})}{g}\!+\!\ldots+\frac{e_{\delta_1\mathcal{A}_n/V_{n-1}}^{(s)}(\mathbf{d})}{g^{s}}\!+\!\mathcal{O}_{n,s}\left(\frac{|\mathbf{d}|^{2s+3}}{g^{s+1}}\right). 
\end{align*}
The error term is polynomial of degree $2s+3$ in $\mathbf{d}$: we gained a degree through multiplication by $(2d_j+1)$. According to Lemma~\ref{cor:polandai}, the sum over $i$ does not change the degree. 
\begin{claim}\label{claim:A2}
The $s$-th term of the expansion $e_{\delta_1\mathcal{A}_n/V_{n-1}}^{(s)}$ verifies:
    \begin{align*}
    e_{\delta_1\mathcal{A}_n/V_{n-1}}^{(s)}(\mathbf{d})
    &=q_{\delta_1\mathcal{A}_n/V_{n-1}}^{(s)}(\mathbf{d})+\sum_{J\subsetneq\{\mathbf{n}\}} q_{\delta_1\mathcal{A}_n/V_{n-1},J}^{(s)}(\mathbf{d})\\
    &=-8\left(\frac{-1}{\pi^{2}}\right)^{s}\sum_{\substack{i_1+\ldots\\+i_n=2s}}\sum_{j=2}^{n}\frac{(2s-1)!!}{i_1!\ldots i_n!}\cdot d_1^{i_1}\ldots d_j^{i_j+1} \ldots d_n^{i_n}+\mathcal{O}_{n,s}\left(|\mathbf{d}|^{2s}\right)
    \end{align*} 
where $q_{\delta_1\mathcal{A}_n/V_{n-1}}^{(s)}$ is a polynomial in $n$ variables of degree $2s+1$ and $q_{\delta_1\mathcal{A}_n/V_{n-1},J}^{(s)}$ are $J$-partial polynomials such that $\mathrm{deg}(q_{\delta_1\mathcal{A}_n/V_{n-1},J}^{(s)})\leq 2s-3$ and $$\mathrm{supp}(q_{\delta_1\mathcal{A}_n/V_{n-1},J}^{(s)})\subset \{\mathbf{d}\in\mathbb{N}^{n-|J|}\lvert |\mathbf{d}|<2s-3\}.$$ 
\end{claim} 
\begin{proof}
    In fact, induction hypothesis implies: 
    \begin{align*}
        e_{n-1}^{(s)}&=q_{n-1}^{(s)}+\sum_{J'\subsetneq\{1\ldots n-1\}} q_{n-1,J'}^{(s)}\\
        &=\left(\frac{-1}{\pi^{2}}\right)^{s}\sum_{\substack{i_1+\ldots+i_{n-1}\\=2s}}\frac{(2s-1)!!}{i_1!\ldots i_n!}\cdot d_1^{i_1}\ldots d_{n-1}^{i_{n-1}}+\mathcal{O}_{n,s}\left(|\mathbf{d}|^{2s-1}\right).
    \end{align*}
    where $q_{n-1}^{(s)}$ is a polynomial of degree $2s$, $q_{n-1,J'}^{(s)}$ are $J'$-partial polynomials of degree $\mathrm{deg}(q_{n-1,J'}^{(s)})\leq 2s-4$ and $\mathrm{supp}(q_{n-1,J'}^{(s)})\subset \{\mathbf{d}\in\mathbb{N}^{n-|J'|}\lvert |\mathbf{d}|<2s-4\}$. 
    For the main polynomial: 
    \begin{align*}
        q_{\delta_1\mathcal{A}_n/V_{n-1}}^{(s)}&=8\;\sum_{j=2}^{n}\sum_{i=0}^{d_0} (a_{i-1}-a_{i})(2d_j+1) \cdot q_{n-1}^{(s)}(d_1\!+\!d_j\!+\!i\!-\!1,d_2\ldots\widehat{d_j},\ldots d_n).
    \end{align*}
    It is a polynomial in $\mathbf{d}$ of degree $2s+1$. Let us compute its dominant terms:
    \begin{align*}
        q_{\delta_1\mathcal{A}_n/V_{n-1}}^{(s)}&=8\;\sum_{j=2}^{n}\sum_{i=0}^{\infty} (a_{i-1}-a_{i})(2d_j+1) \cdot\left(\left(\frac{-1}{\pi^{2}}\right)^{s}\sum_{\substack{i_1'+i_2+\ldots\hat{i_j}\\\ldots+i_n=2s}}\frac{(2s-1)!!}{i_1'!i_2!\ldots i_n!}\right.\\
        &\hspace{2.7cm}\left.\cdot (d_1+d_j+i-1)^{i_1'}\ldots d_n^{i_n}+\mathcal{O}_{n,s}\left((|\mathbf{d}|+i-1)^{2s-1}\right)\right)\\
        &=-8\;\sum_{j=2}^{n}\sum_{i=0}^{\infty} (a_{i-1}-a_{i})\cdot \left(2 \cdot\left(\frac{-1}{\pi^{2}}\right)^{s}\sum_{\substack{i_1'+i_2\ldots\hat{i_j}\\\ldots+i_n=2s}}\frac{(2s-1)!!}{i_1'!i_2!\ldots i_n!}\right.\\
        &\hspace{4cm}\left.\cdot \sum_{i_1+i_j=i_1'}\binom{i_1'}{i_1}d_1^{i_1}d_j^{i_j+1}\ldots d_n^{i_n}+\mathcal{O}_{n,s}\left(|\mathbf{d}|^{2s}\right)\right).
    \end{align*}
    Following Lemma~\ref{cor:polandai},
    \begin{align*}
        q_{\delta_1\mathcal{A}_n/V_{n-1}}^{(s)}&=-8\left(\frac{-1}{\pi^{2}}\right)^{s}\;\sum_{j=2}^{n}\sum_{\substack{i_1+\ldots \\+i_n=2s}}\frac{(2s-1)!!}{i_1!\ldots i_n!}\cdot d_1^{i_1}d_2^{i_2}\ldots d_j^{i_j+1}\ldots d_n^{i_n}+\mathcal{O}_{n,s}\left(|\mathbf{d}|^{2s}\right).
    \end{align*}
    
For the $J$-partial polynomials, there are two distinct cases. In the first case, $1\in J'$ and its behavior is similar to the main polynomial. We obtain $q_{\delta_1\mathcal{A}_n/V_{n-1},J}^{(s)}$ a $J$-partial polynomial, where $$J=\{j\}\cup\left\{\begin{array}{cc}
     k&\text{ for }k\in J'\text{ and }k<j  \\
     k+1&\text{ for }k\in J'\text{ and }k\geq j
\end{array}\right\}.$$ 
The partial polynomial gained one degree, thus $\mathrm{deg}(q_{\delta_1\mathcal{A}_n/V_{n-1},J}^{(s)})\leq 2s-3$, and the support did not change: $$\mathrm{supp}(q_{\delta_1\mathcal{A}_n/V_{n-1},J}^{(s)})=\mathrm{supp}(q_{n-1,J'}^{(s)})\subset \{\mathbf{d}\in\mathbb{N}^{n-|J|}\lvert |\mathbf{d}|<2s-4\}.$$

In the second case, $1\notin J'$. Then, following Proposition~\ref{rem:prodexp}, the product of a polynomial and a $J$-partial polynomial remains a $J$-partial polynomial. As $d_j$ is not a polynomial variable, it does not increase the degree of the partial polynomial. The terms of the sum over $i$ are non-zero only for $i\leq 2s+1$. Summing over this terms kill the dependency in $i$. We obtain $q_{\delta_1\mathcal{A}_n/V_{n-1},J}$ a $J$-partial polynomial where 
$$J=\left\{\begin{array}{cccc}
     k &\text{ for }&k\in J',&k< j\\
     k+1&\text{ for }&k\in J',&k\geq j 
\end{array}\right\}.$$ 
Composition with $(d_1,d_j,i)\mapsto (d_1+d_j+i-1)$ translates the support of the partial polynomial: 
\begin{align*}
    d_1+d_j+i-1+\sum_{\substack{i_j\in
    \{\mathbf{n}\}\setminus J,\\i_j\neq 1,j}}d_{i_j}<2s-4\Rightarrow |\mathbf{d}_{\{\mathbf{n}\}\setminus J}|<2s-4+1-i\leq 2s-3.
\end{align*} 
We observe $\mathrm{deg}(q_{\delta_1\mathcal{A}_n/V_{n-1},J}^{(s)})\leq 2s-3$ and $\mathrm{supp}(q_{\delta_1\mathcal{A}_n/V_{n-1},J}^{(s)})\subset \{\mathbf{d}\in\mathbb{N}^{n-|J|}\lvert |\mathbf{d}|<2s-3\}$.

Thus the proof of Claim~\ref{claim:A2} is complete.

\end{proof}

Therefore the ratio $\delta_1{\mathcal{A}}_{g,n}(\mathbf{d})/V_{g,n}$ admits a $s+1$-th order asymptotic expansion and its $(s+1)$-th coefficient is: 
\begin{align*}
    e_{\delta_1{\mathcal{A}}_{n}/{V_{n}}}^{(s+1)}(\mathbf{d})&=\sum_{r=0}^{s}e_{\delta_1{\mathcal{A}}_{n}/{V_{n-1}}}^{(r)}(\mathbf{d})\cdot\mathfrak{a}_{n}^{(s+1-r)}.
\end{align*}

Each term can be decomposed in a main polynomial and a sum of partial polynomials. Since the factor $\mathfrak{a}_{n}^{(s+1-r)}$ does not depend on $\mathbf{d}$, the sum may be rearranged, yielding a decomposition of $e_{\delta_1{\mathcal{A}}_{n}/{V_{g,n}}}^{(s+1)}$ into a main polynomial and partial polynomials:
\begin{align*}
    e_{\delta_1{\mathcal{A}}_{n}/{V_{n}}}^{(s+1)}(\mathbf{d})&=q_{\delta_1{\mathcal{A}}_{n}/{V_{n}}}^{(s+1)}(\mathbf{d})+\sum_{J\subsetneq\{\mathbf{n}\}}q_{\delta_1{\mathcal{A}}_{n}/{V_{n}},J}^{(s+1)}(\mathbf{d})\\
    &=\mathfrak{a}_{n}^{(0)}\cdot q_{\delta_1{\mathcal{A}}_{n}/{V_{n-1}}}^{(s+1)}(\mathbf{d})+\mathcal{O}_{n,s}\left(|\mathbf{d}|^{2s}\right)\\
    &=\left(\frac{-1}{\pi^{2}}\right)^{s+1}\sum_{\substack{i_1+\ldots\\+i_n=2s}}\sum_{j=2}^{n}\frac{(2s-1)!!}{i_1!\ldots i_n!}\cdot d_1^{i_1}\ldots d_j^{i_j+1} \ldots d_n^{i_n}+\mathcal{O}_{n,s}\left(|\mathbf{d}|^{2s}\right).
\end{align*}
with $q_{\delta_1{\mathcal{A}}_{n}/{V_{g,n}}}^{(s+1)}$ a polynomial of degree $2s+1$ and $q_{\delta_1{\mathcal{A}}_{n}/{V_{g,n}},J}^{(s+1)}$ $J$-partial polynomials of degree $\mathrm{deg}(q_{\delta_1\mathcal{A}_n/V_{n-1},J})\leq 2s-3$ and $\mathrm{supp}(q_{\delta_1\mathcal{A}_n/V_{n-1},J}^{(s)})\subset \{\mathbf{d}\in\mathbb{N}^{n-|J|}\lvert |\mathbf{d}|<2s-3\}$.

\subsection{Contribution of the second term $\frac{\delta_1{\mathcal{B}}_{g,n}(\mathbf{d})}{V_{g,n}} $.}\label{sec:B}

The procedure is similar for the term $\frac{\delta_1{\mathcal{B}}_{g,n}(\mathbf{d})}{V_{g,n}}$. We write it as a product: 
\begin{align*}
    \frac{\delta_1{\mathcal{B}}_{g,n}(\mathbf{d})}{V_{g,n}} =& \frac{1}{4\pi^2 (2g-3+n)}\cdot \frac{4 \pi^2 (2g-3+n)V_{g-1,n+1}}{V_{g-1,n+2}}\cdot \frac{V_{g-1,n+2}}{V_{g,n}}\cdot\frac{\delta_1{\mathcal{B}}_{g,n}(\mathbf{d})}{V_{g-1,n+1}},\\
    \frac{\delta_1{\mathcal{B}}_{g,n}(\mathbf{d})}{V_{g-1,n+1}} =& 16\sum_{i=0}^{\infty}\sum_{\substack{k_1+k_2=\\d_1+i-2}}(a_{i-1}-a_{i}) \frac{[\tau_{k_1}\tau_{k_2}\tau_{d_{2}}\ldots \tau_{d_{n}}]_{g-1,n+1}}{V_{g-1,n+1}}.
\end{align*}
As for the first term, we first study separately the expansion of $$(\ast):=\frac{1}{4\pi^2 (2g-3+n)}\cdot \frac{4 \pi^2 (2g-3+n)V_{g-1,n+1}}{V_{g-1,n+2}}\cdot \frac{V_{g-1,n+2}}{V_{g,n}}$$
and the expansion of $$\frac{\delta_1{\mathcal{B}}_{g,n}(\mathbf{d})}{V_{g-1,n+1}}.$$

The factor $(\ast)$ admits a $(s+1)$-th order asymptotic expansion according to Remark~\ref{rem:prodexp} and because 
$$\frac{1}{4\pi^2 (2g-3+n)}=\mathcal{O}_{n}\left(\frac{1}{g}\right).$$ 

We denote $\mathfrak{b}_{n}^{(r)}$ the coefficients of this expansion:
\begin{align*}
    (\ast)=\frac{\mathfrak{b}_{n}^{(1)}}{g}+\ldots +\frac{\mathfrak{b}_{n}^{(s+1)}}{g^{s+1}}+\mathcal{O}_{n,s}\left(\frac{1}{g^{s+2}}\right)
\end{align*}
with 
\begin{align*}
    \mathfrak{b}_n^{(r)}=\frac{1}{8\pi^{2}}\sum_{\substack{r_1+r_2\\+r_3=r}}&\left(\frac{n-3}{2}\right)^{r_1}\cdot a_{n+1}^{(r_2)}\cdot b_{n+1}^{(r_3)}
\end{align*}
and, in particular, $\mathfrak{b}_n^{(1)}=\frac{1}{8\pi^{2}}.$

On the other hand, $\frac{\delta_1{\mathcal{B}}_{g,n}(\mathbf{d})}{V_{g-1,n+1}}$ admits a $s$-th order expansion which depends on the expansion of $\frac{[\tau_{k_1}\tau_{k_2}\tau_{d_{2}}\ldots \tau_{d_{n}}]_{g-1,n+1}}{V_{g-1,n+1}}$:
\begin{align*}
    \frac{\delta_1{\mathcal{B}}_{g,n}(\mathbf{d})}{V_{g-1,n+1}}&=16\sum_{i=0}^{\infty}\sum_{\substack{k_1+k_2=\\d_1+i-2}}(a_{i-1}-a_{i}) \cdot e_{n+1}^{(0)}(k_1,k_2,d_2\ldots d_n)\\
    &+\frac{1}{g}\left(16\sum_{i=0}^{\infty}\sum_{\substack{k_1+k_2=\\d_1+i-2}}(a_{i-1}-a_{i}) \cdot e_{n+1}^{(1)}(k_1,k_2,d_2\ldots d_n)\right)\\
    &+\ldots \\
    &+\frac{1}{g^{s}}\left(16\sum_{i=0}^{\infty}\sum_{\substack{k_1+k_2=\\d_1+i-2}}(a_{i-1}-a_{i}) \cdot e_{n+1}^{(s)}(k_1,k_2,d_2\ldots d_n)\right)\\
    &+\left(16\cdot\sum_{i=0}^{\infty}\sum_{\substack{k_1+k_2=\\d_1+i-2}}(a_{i-1}-a_{i}) \cdot\mathcal{O}_{n,s}\left(\frac{|\mathbf{d}|^{2s+2}}{g^{s+1}}\right)\right)\\
    &=e_{\delta_1\mathcal{B}_n/V_{n+1}}^{(0)}(\mathbf{d})+\frac{e_{\delta_1\mathcal{B}_n/V_{n+1}}^{(1)}(\mathbf{d})}{g}\!+\!\ldots\!+\!\frac{e_{\delta_1\mathcal{B}_n/V_{n+1}}^{(s)}(\mathbf{d})}{g^{s}}+\mathcal{O}_{n,s}\left(\frac{|\mathbf{d}|^{2s+3}}{g^{s+1}}\right).
\end{align*}
By induction hypothesis, the error term of $\frac{[\tau_{k_1}\tau_{k_2}\tau_{d_{2}}\ldots \tau_{d_{n}}]_{g-1,n+1}}{V_{g-1,n+1}}$ was polynomial in $(k_1,k_2,d_2,\ldots d_n)$ of degree $2s+2$. According to Remark~\ref{rem:convo}, the sum over $k_1+k_2=d_1+i-2$ gives a polynomial in $(d_1+i-2, d_2,\ldots d_n)$ of degree $2s+3$. The Lemma~\ref{lem:sumoveri} ensures that the sum over $i$ leaves the degree unchanged and we obtain a polynomial bound in $\mathbf{d}$ of degree $2s+3$.

\begin{claim}\label{claim:B2}
The $s$-th term of this expansion $e_{\delta_1\mathcal{B}_n/V_{n+1}}^{(s)}$ verifies:
\begin{align*}
e_{\delta_1\mathcal{B}_n/V_{n+1}}^{(s)}(\mathbf{d})
&=q_{\delta_1\mathcal{B}_n/V_{n+1}}^{(s)}(\mathbf{d})+\sum_{J\subsetneq\{\mathbf{n}\}} q_{\delta_1\mathcal{B}_n/V_{n+1},J}^{(s)}(\mathbf{d})\\
&=-8\left(\frac{-1}{\pi^{2}}\right)^{s}\sum_{i_1+...\,+i_n=2s}\frac{(2s-1)!!}{i_1!...\, i_n!}\cdot d_1^{i_1+1} d_2^{i_2}\ldots d_n^{i_n}+\mathcal{O}_{n,s}\left(|\mathbf{d}|^{2s}\right)
\end{align*}
with $q_{\delta_1\mathcal{B}_n/V_{n+1}}^{(s)}$ a polynomial of degree $2s+1$ and $q_{\delta_1\mathcal{B}_n/V_{n+1},J}^{(s)}$ a $J$-partial polynomial of degree $\mathrm{deg}(q_{\delta_1\mathcal{B}_n/V_{n+1},J})\leq 2s-3$ and $\mathrm{supp}(q_{\delta_1\mathcal{B}_n/V_{n+1},J}^{(s)})\subset \{\mathbf{d}\in\mathbb{N}^{n-|J|}\lvert |\mathbf{d}|<2s-2\}$.

\end{claim} 
\begin{proof}[Proof of Claim~\ref{claim:B2}]

By induction hypothesis: 
$$e_{n+1}^{(s)}(k_1,k_2,d_2,\ldots d_n)=q_{n+1}^{(s)}(k_1,k_2,d_2,\ldots d_n)+\sum_{J'\subsetneq\{1\ldots n+1\}} q_{n+1,J'}^{(s)}(k_1,k_2,d_2,\ldots d_n)$$ 
where $q_{n+1}^{(s)}$ is a polynomial of degree $2r$, $q_{n+1,J}^{(s)}$ is a $J$-partial polynomial of degree $\mathrm{deg}(q_{n+1,J}^{(s)})\leq 2s-4$ and $\mathrm{supp}(q_{n+1,J}^{(s)})\subset \{\mathbf{d}\in\mathbb{N}^{n-|J|}\lvert |\mathbf{d}|<2s-4\}$.
    
Let us first study the main polynomial: following Remark~\ref{rem:convo} and Corollary~\ref{cor:polandai},
\begin{align*}
   q_{\delta_1\mathcal{B}_n/V_{n+1}}^{(s)}&=16\sum_{i=0}^{\infty}\sum_{\substack{k_1+k_2=\\d_1+i-2}}(a_{i-1}-a_{i}) \cdot q_{n+1}^{(s)}(k_1,k_2,d_2\ldots d_n)
\end{align*}
is a polynomial in $\mathbf{d}$ of degree $2s+1$. Let us compute its dominant term:
\begin{align*}
   q_{\delta_1\mathcal{B}_n/V_{n+1}}^{(s)}&=16\sum_{i=0}^{\infty}\sum_{\substack{k_1+k_2=\\d_1+i-2}}(a_{i-1}-a_{i}) \cdot\left(\frac{-1}{\pi^{2}}\right)^{s}\left(\sum_{\substack{i_{k_1}+i_{k_2}+i_2+\\\ldots+i_n=2s}}\frac{(2s-1)!!}{i_{k_1}!i_{k_2}!i_2!\ldots i_n!}\right.\\
   &\hspace{2cm}\left.\cdot k_1^{i_{k_1}}k_2^{i_{k_2}}d_2^{i_2}\ldots d_n^{i_n}+\mathcal{O}_{n,s}\left(|k_1+k_2+d_2+\ldots+d_n|^{2s}\right)\right)
\end{align*}

By Remark~\ref{rem:convo},

\begin{align*}
    \sum_{\substack{k_1+k_2=\\d_1+i-2}}k_1^{i_{k_1}}k_2^{i_{k_2}}&=\frac{i_{k_1}!i_{k_2}!}{(i_{k_1}+i_{k_2}+1)!}(d_1+i-2)^{i_{k_1}+i_{k_2}+1}+\mathcal{O}\left((d_1+i-2)^{i_{k_1}+i_{k_2}}\right).
\end{align*}

Then:
\begin{align*}
   q_{\delta_1\mathcal{B}_n/V_{n+1}}^{(s)}&=16\left(\frac{-1}{\pi^{2}}\right)^{s}\sum_{\substack{i_{k_1}+i_{k_2}+i_2+\\\ldots+i_n=2s}}\cdot\frac{(2s-1)!!}{(i_{k_1}+i_{k_2}+1)!i_2!\ldots i_n!}d_2^{i_2}\ldots d_n^{i_n}\\
   &\hspace{2.2cm}\cdot \left(\sum_{i=0}^{\infty}(a_{i-1}-a_{i})(d_1+i-2)^{i_{k_1}+i_{k_2}+1}\right)+\mathcal{O}_{n,s}\left(|\mathbf{d}|^{2s}\right)
\end{align*}
According to Corollary~\ref{cor:polandai}:
\begin{align*}
   q_{\delta_1\mathcal{B}_n/V_{n+1}}^{(s)}&=-8\left(\frac{-1}{\pi^{2}}\right)^{s}\sum_{\substack{i_{k_1}+i_{k_2}+i_2+\\\ldots+i_n=2s}}\frac{(2s-1)!!}{(i_{k_1}+i_{k_2}+1)!i_2!\ldots i_n!}\\
   &\hspace{5.35cm}\cdot d_1^{i_{k_1}+i_{k_2}+1}d_2^{i_2}\ldots d_n^{i_n}+\mathcal{O}_{n,s}\left(|\mathbf{d}|^{2s+1}\right)\\
   &=-8\left(\frac{-1}{\pi^{2}}\right)^{s}\sum_{\substack{i_1+i_2+\ldots\\+i_n=2s}}\frac{(2s-1)!!}{i_1!i_2!\ldots i_n!}\cdot d_1^{i_1+1}d_2^{i_2}\ldots d_n^{i_n}+\mathcal{O}_{n,s}\left(|\mathbf{d}|^{2s+1}\right).
\end{align*}
    
We now consider the contribution of the partial polynomials $q_{n+1,J'}^{(s)}$. There are three cases. First case: both the first and second index belong to $J'$, ie $1,2\in J'$. Then:  
$$16\sum_{i=0}^{\infty}\sum_{\substack{k_1+k_2=\\d_1+i-2}}(a_{i-1}-a_{i}) \cdot q_{n+1,J'}^{(s)}(k_1,k_2,d_2\ldots d_n)$$
contributes to $q_{\delta_1\mathcal{B}_n/V_{n+1},J}^{(s)}$ a $J$-partial polynomial where $J=\{1,j-1\text{ for }j\in J'\}$, with degree strictly less than $2s-3$ and support included in $\{\mathbf{d}\in\mathbb{N}_+^{n-|J|}\lvert|\mathbf{d}|<2s-4\}$. The proof follows the same step as for the main polynomial. 

Second case: neither the first nor the second index belongs to $J'$, ie $1,2\notin J'$, $$16\sum_{i=0}^{\infty}\sum_{\substack{k_1+k_2=\\d_1+i-2}}(a_{i-1}-a_{i}) \cdot q_{n+1,J'}^{(s)}(k_1,k_2,d_2\ldots d_n)$$
contributes to $q_{\delta_1\mathcal{B}_n/V_{n+1},J}^{(s)}$ a $J$-partial polynomial where $J=\{j-1\text{ for }j\in J'\}$, with degree strictly less than $2s-4$ and support included in $\{\mathbf{d}\in\mathbb{N}_+^{n-|J|}\lvert|\mathbf{d}|<2s-2\}$. The support of the partial polynomial is translated: 
\begin{align*}
    k_1+k_2+\sum_{j\notin J'}d_{j-1}<2s-4&\Rightarrow d_1+j-2+\sum_{j+1\notin J'}d_j<2s-4\\&\Rightarrow \sum_{j\notin J}d_j<2s-2-i\leq 2s-2.
\end{align*}
For $i>2s-2$, the contribution is zero. Then we are left with a finite sum of $J$-partial polynomials, which is still a $J$-partial polynomial by Remark~\ref{rem:partpol}. The support is maximal for $i=0$: it is $\{\mathbf{d}\in\mathbb{N}_+^{n-|J|}\lvert|\mathbf{d}|<2s-2\}$. 

Third case: only one of the two first indices belongs to $J'$, ie $1\in J'$, $2\notin J'$ or $1\notin J'$, $2\in J'$, $$16\sum_{i=0}^{\infty}\sum_{\substack{k_1+k_2=\\d_1+i-2}}(a_{i-1}-a_{i}) \cdot q_{n+1,J'}^{(s)}(k_1,k_2,d_2\ldots d_n)$$
contributes to two partial polynomials $q_{\delta_1\mathcal{B}_n/V_{n+1},J_1}^{(s)}$ where $J_1=\{1,j-1\text{ for }j\in J'\}$ and  $q_{\delta_1\mathcal{B}_n/V_{n+1},J_2}^{(s)}$ where $J_2=\{j-1\text{ for }j\in J'\}$. The contribution to the $J_1$-partial polynomial has degree strictly less than $2s-4$ and support $\{\mathbf{d}\in\mathbb{N}_+^{n-|J|}\lvert|\mathbf{d}|<2s-4\}$ and the contribution to the $J_2$-partial polynomial has degree strictly less than $2s-4$ and support $\{\mathbf{d}\in\mathbb{N}_+^{n-|J|}\lvert|\mathbf{d}|<2s-4\}$. 

Let us prove the statement for $1\notin J$, $2\in J$. When $d_1+i\geq 2r+2-\sum_{j+1\notin J'}d_j$, the convolution over $k_1+k_2=d_1+i-2$ is cut by $k_1<2s-4$. Thus we have a finite sum of $J'$-partial polynomials $q_{n+1,J'}^{(0)}(k,d_1\!+\!i\!-\!2\!-\!k,d_2\ldots d_n)$. The sum kills the dependence in $k$ and we obtain a $J_1$-partial polynomial, whose degree is still strictly less than $2s-4$. We handle the sum over $i$ as we did for the other terms. However, this is exact only for $d_1+i>2s-4+2-\sum_{j+1\notin J'}d_j$. When $d_1+i < 2s-4+2-\sum_{j+1\notin J'}d_j$, $k$ does not take all possible values. We add a correction, which can be written as a $J_2$-partial polynomial with support $\{\mathbf{d}\in\mathbb{N}_+^{n-|J_2|}\lvert|\mathbf{d}|<2s-2\}$. 

The proof of Claim~\ref{claim:B2} is complete.

\end{proof} 
   
The asymptotic expansion of $\frac{\delta_1{\mathcal{B}}_{g,n}(\mathbf{d})}{V_{g,n}}$ is expressed as the product of the two expansions we computed and therefore it admits a $(s+1)$-th expansion thanks to Remark~\ref{rem:prodexp}: 

\begin{align*}
    \frac{\delta_1{\mathcal{B}}_{g,n}(\mathbf{d})}{V_{g,n}}&=\frac{1}{4\pi^2 (2g-3+n)}\cdot \frac{4 \pi^2 (2g-3+n)V_{g-1,n+1}}{V_{g-1,n+2}}\cdot \frac{V_{g-1,n+2}}{V_{g,n}}\cdot\frac{\delta_1{\mathcal{B}}_{g,n}(\mathbf{d})}{V_{g-1,n+1}}\\
    &=\left(\frac{\mathfrak{b}_{n}^{(1)}}{g}+..\! +\frac{\mathfrak{b}_{n}^{(s+1)}}{g^{s+1}}+\mathcal{O}_{n,s}\left(\frac{1}{g^{s+2}}\right)\right)\\
    &\quad\cdot \left(e_{\delta_1\mathcal{B}_n/V_{n+1}}^{(0)}(\mathbf{d})\!+\!\frac{e_{\delta_1\mathcal{B}_n/V_{n+1}}^{(1)}(\mathbf{d})}{g}\!+..\!+\frac{e_{\delta_1\mathcal{B}_n/V_{n+1}}^{(s)}(\mathbf{d})}{g^{s}}\!+\!\mathcal{O}_{n,s}\left(\frac{|\mathbf{d}|^{2s+3}}{g^{s+1}}\right)\right)\\
    &=\frac{e_{\delta_1\mathcal{B}_n/V_{n}}^{(1)}(\mathbf{d})}{g}+\ldots +\frac{e_{\delta_1\mathcal{B}_n/V_{n}}^{(s+1)}(\mathbf{d})}{g^{s+1}}+\mathcal{O}_{n,s}\left(\frac{|\mathbf{d}|^{2s+3}}{g^{s+2}}\right)
\end{align*} 

where 
\begin{align*}
    e_{\delta_1\mathcal{B}_n/V_{n}}^{(s+1)}(\mathbf{d})&=\sum_{r=0}^{s}e_{\delta_1\mathcal{B}_n/V_{n+1}}^{(r)}(\mathbf{d})\mathfrak{b}_{n}^{(s+1-r)}.
\end{align*}
We rearrange the decompositions in polynomials and partial polynomials of $e_{\delta_1\mathcal{B}_n/V_{n+1}}^{(r)}$ for $r\leq s$ to obtain the following decomposition: 
\begin{align*}
    e_{\delta_1\mathcal{B}_n/V_{n}}^{(s+1)}=q_{\delta_1\mathcal{B}_n/V_{n}}^{(s+1)}+\sum_{J\subsetneq\{\mathbf{n}\}} q_{\delta_1\mathcal{B}_n/V_{n},J}^{(s+1)}
\end{align*} 
with $q_{\delta_1\mathcal{B}_n/V_{n}}^{(s)}$ a polynomial of degree $2s+1$ and $q_{\delta_1\mathcal{B}_n/V_{n},J}^{(s)}$ a $J$-partial polynomial of degree $\mathrm{deg}(q_{\delta_1\mathcal{B}_n/V_{n},J}^{(s+1)})\leq 2s-3$ and with support included in $\{\mathbf{d}\in\mathbb{N}_+^{n-|J|}\lvert|\mathbf{d}|<2s-2\}$. The dominant term is:
\begin{align*}
    e_{\delta_1\mathcal{B}_n/V_{n}}^{(s+1)}(\mathbf{d})&=q_{\delta_1\mathcal{B}_n/V_{n+1}}^{(s)}(\mathbf{d})\cdot\mathfrak{b}_{n}^{(1)}+\mathcal{O}_{n,s}\left(|\mathbf{d}|^{2s-1}\right)\\
    &=\frac{-1}{\pi^{2}}\cdot\left(\frac{-1}{\pi^{2}}\right)^{s}\sum_{i_1+\ldots+i_n=2s}\frac{(2s-1)!!}{(i_1+1)!\ldots i_n!}\cdot d_1^{i_1+1} d_2^{i_2}\ldots d_n^{i_n}+\mathcal{O}_{n,s}\left(|\mathbf{d}|^{2s}\right).
\end{align*}

\subsection{Contribution of the third term $\frac{\delta_1\mathcal{C}_{g,n}(\mathbf{d})}{V_{g,n}}$.}

Until now we have seen that the decomposition in polynomials and partial polynomials is well preserved in the expansion of $\frac{\delta_1\mathcal{A}_n}{V_{g,n}}$ and $\frac{\delta_1\mathcal{B}_n}{V_{g,n}}$ but we don't see how this partial polynomials appear. In fact they come from the third term in the topological recursion. 

\begin{claim}\label{claim:C}
The derivative of the third term of the topological recursion $\frac{\delta_1\mathcal{C}_{g,n}(\mathbf{d})}{V_{g,n}}$ admits an expansion: 
\begin{align*}
    \frac{\delta_1\mathcal{C}_n(\mathbf{d})}{V_{g,n}}=e_{\delta_1\mathcal{C}_n}^{(0)}(\mathbf{d})+\ldots +\frac{e_{\delta_1\mathcal{C}_n}^{(s+1)}(\mathbf{d})}{g^{s+1}}+\mathcal{O}_{n,s}\left(\frac{|\mathbf{d}|^{2s}}{g^{s+2}}\right)
\end{align*}
where
$$e_{\delta_1\mathcal{C}_n}^{(s+1)}=q_{\delta_1\mathcal{C}_n}^{(s+1)}+\sum_{J\subsetneq\{\mathbf{n}\}}q_{\delta_1\mathcal{C}_n,J}^{(s+1)}$$
with $q_{\delta_1\mathcal{C}_n}^{(s+1)}$ polynomial of degree $2s-2$ and $q_{\delta_1\mathcal{C}_n,J}^{(s+1)}$ $J$-partial polynomials of degree at most $2s-2$ and support $\{\mathbf{d}\in\mathbb{N}^{n-|J|}\lvert |\mathbf{d}|<2s-2\}.$
\end{claim}

Before we go into the proof, we give a short informal explanation of the apparition of partial polynomials and the lower degree contribution of $\delta_1\mathcal{C}_n(\mathbf{d})/V_{g,n}$. Topologically, the third term corresponds to the case where by taking out a pair of pants the surface is disconnected, see $(C)$ on figure~\ref{fig:recursion}. It is much more rare than to keep the surface connected, and it is even more unlikely that by doing so we disconnect the surface in two sub-surfaces of more or less equal size. Most probably we get on one side a one handed torus or a single pair of pants, and the rest of the topology on the other side. As $g$ goes to infinity, most probably only one of the two subsurface also has its genus going to infinity. Only these cases contribute to the asymptotic expansion. By separating the surface, we have created a partition of boundary components in $I\sqcup J=\{2,\ldots n\}$. Let $I$ be the indices of the boundary components of the subsurface which remains of bounded topological complexity. Now if we look at the intersection numbers associated to the boundary components in $I$, they take only finitely many values, because the dimension of the modular space they live in is finite. They give rise to the functions of compact support in $J$-partial polynomials. 
\begin{proof}[Proof of Claim~\ref{claim:C}]

The term $\frac{\delta_1{\mathcal{C}}_{g,n}(\mathbf{d})}{V_{g,n}}$ is defined by:
    
%FIGURES !!!!!
\begin{align*}
        \frac{\delta_1{\mathcal{C}}_{g,n}(\mathbf{d})}{V_{g,n}} &= 16\sum_{\genfrac{}{}{0pt}{}{g_1+g_2=g}{I\amalg J=\{2,\ldots, n\}}} \sum_{i=0}^{\infty} \sum_{\genfrac{}{}{0pt}{}{k_{1}+k_{2}=}{=i+d_{1}-2}} (a_{i-1}-a_{i}) \\
        &\hspace*{4cm}\cdot\frac{\left[\tau_{k_{1}} \prod_{i\in I } \tau_{d_{i}}\right]_{g_1,|I|+1} \cdot \left[\tau_{k_{2}} \prod_{j\in J} \tau_{d_{j}}\right]_{g_2,|J|+1}}{V_{g,n}}\\
\end{align*}

Only a finite number of partitions of the genus and the boundary components contribute to the $s$-th order expansion:  

\begin{lem}\label{lem:sumpartition}
\begin{align*}
    \frac{\delta_1{\mathcal{C}}_{g,n}(\mathbf{d})}{V_{g,n}}&= 32\cdot\sum_{\substack{g_1+g_2=g\\I\amalg J=\{2,\ldots, n\}\\2g_1+|I|+1\leq s}}\frac{V_{g-g_1,n-|I|}}{V_{g,n}} \cdot \sum_{i=0}^{\infty} \sum_{\genfrac{}{}{0pt}{}{k_{1}+k_{2}}{=i+d_{1}-2}}(a_{i-1}-a_{i})\\
    &\hspace*{2cm}\cdot\left[\tau_{k_{1}} \prod_{i\in I } \tau_{d_{i}}\right]_{g_1,|I|+1}\cdot\frac{\left[\tau_{k_{2}} \prod_{j\in J} \tau_{d_{j}}\right]_{g-g_1,|J|+1}}{V_{g-g_1,\lvert J\rvert+1}}+\mathcal{O}_{n,s}\left({\frac{1}{g^{s}}}\right).
\end{align*}
\end{lem}

\begin{proof}[Proof of Lemma~\ref{lem:sumpartition}]

This is a direct consequence of Lemma~\ref{lem:sumratioMZ}. In fact: 
\begin{equation*}
    \frac{\left[\tau_{k_{1}} \prod_{i\in I } \tau_{d_{i}}\right]_{g_1,|I|+1} \cdot \left[\tau_{k_{2}} \prod_{i\in J} \tau_{d_{i}}\right]_{g_2,|J|+1}}{V_{g,n+1}}\leq\frac{V_{g_1,\lvert I\rvert +1} \cdot V_{g_2,\lvert J\rvert+1}}{V_{g,n+1}}.
\end{equation*}

Therefore, following Lemma~\ref{lem:sumratioMZ}: 
\begin{equation*}
     \sum_{I\sqcup J=\{\mathbf{n}\}}\sum_{\genfrac{}{}{0pt}{2}{g_1+g_2=g}{2g_i+n_i+1\geq s+1}}\frac{\left[\tau_{k_{1}}\cdot \prod_{i\in I } \tau_{d_{i}}\right]_{g_1,|I|+1} \cdot \left[\tau_{k_{2}} \prod_{i\in J} \tau_{d_{i}}\right]_{g_2,|J|+1}}{V_{g,n+1}}= \mathcal{O}_{n,s}\left(\frac{1}{g^{s+1}}\right)
\end{equation*}
where $n_1 = \lvert I \rvert$ and $n_2 = \lvert J\lvert$.

Finally, the expansion of the ratio is:
$$\frac{V_{g,n+1}}{V_{g,n}}=\mathcal{O}_{n}\left(g\right).$$

To compute the expansion of $\frac{\delta_1{\mathcal{C}}_{g,n}(\mathbf{d})}{V_{g,n}}$ up to order $s+1$, we only need to consider the cases for which either $2g_1+\lvert I\rvert +1\leq s$ or $2g_2+\lvert J\rvert +1\leq s$. By symmetry we consider only $2g_1+\lvert I\rvert +1\leq s$ and add a factor $2$. The proof of Lemma~\ref{lem:sumpartition} is complete.
\end{proof}

\begin{rem}
As a direct consequence, the third term does not contribute to the first order expansion: 
\begin{equation*}
    \frac{\delta_1{\mathcal{C}}_{g,n}(\mathbf{d})}{V_{g,n}}=\mathcal{O}_{n}\left(\frac{1}{g^{2}}\right).
\end{equation*}
In fact, in the topological recursion, the sub-surfaces $S_i$ have to be hyperbolic, ie $2g_i+n_i+1-2\geq 1$. Therefore one does not need to take into account the third term to compute the first-order expansion \cite{AM22}. 
\end{rem}

Let us fix a partition of $g$ and of the boundary components. On one hand, one can compute the expansion of the factors that do not depend on $\mathbf{d}$: 
$$\frac{V_{g-g_1,n-|I|}}{V_{g,n}}.$$

It can be written as a telescoping product:
\begin{align*}
\frac{V_{g-g_1,n-|I|}}{V_{g,n}}= \prod_{j=-|I|}^{2g_1-1}& \frac{4 \pi^2 (2(g-g_1)+n-j+1)V_{g-g_1,n+j}}{V_{g-g_1,n+j+1}}\\
\prod_{j=1}^{g_1} \frac{V_{g-j,n+2j}}{V_{g-j+1,n+2j-2}} &\cdot  \prod_{j=-|I|}^{2g_1-1} \frac{1}{(4 \pi^2 (2(g-g_1)+n-j+1))}.
\end{align*}
  
The last product has $2g_1+|I|\geq 2$ terms. Thus:
$$\frac{V_{g-g_1,n-|I|}}{V_{g,n}}=\mathcal{O}_{n,g_1,|I|}\left({\frac{1}{g^{2g_1+|I|}}}\right).$$
Therefore, knowing expansions of $V_{g,n}/V_{g-1,n+2}$ and $4\pi^2 (2g-2+n)V_{g,n}/V_{g,n+1}$ up to order $s+1-2g_1-|I|\leq s-1$ suffice to determine the expansion of $V_{g-g_1,n-|I|}/V_{g,n}$ up to order $s+1$:

$$\frac{V_{g-g_1,n-|I|}}{V_{g,n}}=\frac{\mathfrak{c}_{n,g_1,n_1}^{(2g_1+n_1)}}{g^{2g_1+|I|}}+\ldots+\frac{\mathfrak{c}_{n,g_1,n_1}^{(s+1)}}{g^{s+1}}+\mathcal{O}_{n,g_1,|I|}\left(\frac{1}{g^{s+2}}\right).$$

On the other hand, we substitute to $\left[\tau_{k_{2}} \prod_{j\in J} \tau_{d_{j}}\right]_{g_2,n_2+1}/V_{g_2,n_2+1}$ its expansion up to order $s+1-2g_1-n_1\leq s-1$ and obtain:
\begin{align*}
    \sum_{i=0}^{\infty} \sum_{\genfrac{}{}{0pt}{}{k_{1}+k_{2}}{=i+d_{1}-2}}(a_{i-1}-&a_{i})
    \cdot\left[\tau_{k_{1}} \prod_{i\in I } \tau_{d_{i}}\right]_{g_1,|I|+1}\cdot\frac{\left[\tau_{k_{2}} \prod_{j\in J} \tau_{d_{j}}\right]_{g-g_1,|J|+1}}{V_{g-g_1,\lvert J\rvert+1}}=\\
    &e_{\delta\mathcal{C},g_1,J}^{(0)}(\mathbf{d})+\frac{e_{\delta\mathcal{C},g_1,J}^{(1)}(\mathbf{d})}{g}+\ldots +\frac{e_{\delta\mathcal{C},g_1,J}^{(s-1)}(\mathbf{d})}{g^{s-1}}+\mathcal{O}_{n,g_1,|I|}\left(\frac{|\mathbf{d}|^{2s}}{g^{s}}\right)
\end{align*}

The error term of this expansion is polynomial of degree $2s$ in $\mathbf{d}$. By induction hypothesis, the error term of the $s+1-2g_1-n_1\leq s-1$-th order expansion of $\frac{\left[\tau_{k_{2}} \prod_{j\in J} \tau_{d_{j}}\right]_{g_2,n_2+1}}{V_{g_2,n_2+1}}$ is polynomial in $(k_2,\mathbf{d}_J)$ of degree $2(s+1-2g_1-n_1+1)\leq 2s$.
Neither the sum over $k_1+k_2$ nor over $i$ change the degree of the polynomial bound. 

\begin{claim}\label{claim:C1}

The $s-1$-th term of this expansion $e_{\delta\mathcal{C},g_1,J}^{(s-1)}$ verifies: 
\begin{align*}
    e_{\delta\mathcal{C},g_1,J}^{(s-1)}&=q_{\delta\mathcal{C},g_1,J}^{(s-1)}+\sum_{J'\subsetneq J}q_{\delta\mathcal{C},g_1,J'}^{(s-1)}
\end{align*}
with $q_{\delta\mathcal{C},g_1,J}^{(s-1)}$ a polynomial if $J=\{\mathbf{n}\}$ and a $J$-partial polynomial else, both of degree $2s-2$, and $q_{\delta\mathcal{C},g_1,J'}^{(s-1)}$ a $J'$-partial polynomial of degree at most $2s-6$. 
\end{claim}

\begin{proof}[Proof of Claim~\ref{claim:C1}]
By induction hypothesis we have for any $r\leq s$ $$e_{|J|+1}^{(r)}(k_2,\mathbf{d}_J)=q_{|J|+1}^{(r)}(k_2,\mathbf{d}_J)+\sum_{\substack{\tilde{J'}\subsetneq\{1\ldots |J|+1\}}}q_{|J|+1,J'}(k_2,\mathbf{d}_J).$$

We recall that intersection numbers on $\mathcal{M}_{g,n}$ are non-zero if and only if $|\mathbf{d}|\leq 3g-3+n$, see Proposition~\ref{prop:internbMZ}. On the component of genus $g_1$, with boundary components indexed by $I$, one has $2g_1+|I|+1\leq s$ and: 
$$k_1+\sum_{i\in I}d_i\leq 3g_1-3+|I|+1=3g_1+|I|-2 < 2s-4.$$

Else intersection numbers vanish: for any $k_1+\sum_{i\in I}d_i > 3g_1-3+|I|+1$, $$\left[\tau_{k_{1}} \prod_{i\in I } \tau_{d_{i}}\right]_{g_1,|I|+1}=0.$$ 

Thus $\left[\tau_{k_1}\prod_{i\in I}\tau_{d_i}\right]\cdot q_{|J|+1}^{(s-1)}$ is a $\tilde{J}$-partial polynomial in $\tilde{\mathbf{d}}=(k_1,k_2,d_2,\ldots d_n)$ with $\tilde{\mathbf{d}}_{\tilde{J}}=\{k_2,\mathbf{d}_J\}$. It has degree $2s-2$ and support
$\left\{k_1,\mathbf{d}_I\,|\,k_1+|\mathbf{d}_I|<2s-4\right\}$. 

The sum over $k_1+k_2$ and the sum over $i$ are handled exactly as the partial polynomials for the second term $\mathcal{B}_{g,n}$, when either $1$ or $2\notin J$, see the proof of Claim~\ref{claim:B2} in Section~\ref{sec:B}. One obtains 
\begin{align*}
    q_{\delta\mathcal{C},g_1,J}^{(s-1)}(\mathbf{d}_J)&=\sum_{i=0}^{\infty} \sum_{\genfrac{}{}{0pt}{}{k_{1}+k_{2}}{=i+d_{1}-2}}(a_{i-1}-a_{i})&
    \cdot\left[\tau_{k_{1}} \prod_{i\in I } \tau_{d_{i}}\right]_{g_1,|I|+1}\cdot q_{|J|+1}^{(s-1)}(k_2,\mathbf{d}_J)
\end{align*}
as in Claim~\ref{claim:C1}.

For partial polynomials $q_{|J|+1,\tilde{J'}}(k_2,\mathbf{d}_J)$, the procedure is similar to the proof of Claim~\ref{claim:B2}. One can naturally define $\tilde{\mathbf{d}}_{\tilde{J'}}$ as the subset of $(k_2,\mathbf{d}_J)$ of indices $\tilde{J'}$. Then 
\begin{align*}
    q_{\delta\mathcal{C},g_1,J'}^{(s-1)}(\mathbf{d}_J)&=\sum_{i=0}^{\infty} \sum_{\genfrac{}{}{0pt}{}{k_{1}+k_{2}}{=i+d_{1}-2}}(a_{i-1}-a_{i})&
    \cdot\left[\tau_{k_{1}} \prod_{i\in I } \tau_{d_{i}}\right]_{g_1,|I|+1}\cdot q_{|J|+1,\tilde{J'}}^{(s-1)}(k_2,\mathbf{d}_J)
\end{align*}
is a $J'$-partial polynomial with $J'$ such that $\tilde{\mathbf{d}}_{\tilde{J'}}=(k_2,\mathbf{d}_{J'})$. It has degree at most $2s-6$ and support $\{\mathbf{d}_{\mathbf{n}\setminus J'}||\mathbf{d}_{\mathbf{n}\setminus J'}|<2s-4\}$.

The proof of Claim~\ref{claim:C1} is complete.
\end{proof}

Summing over the partitions, the expansion of $\delta_1\mathcal{C}_n(\mathbf{d})/V_{g,n}$ is then: 

\begin{align*}
    \frac{\delta_1\mathcal{C}_n(\mathbf{d})}{V_{g,n}}=e_{\delta_1\mathcal{C}_n}^{(0)}(\mathbf{d})+\ldots +\frac{e_{\delta_1\mathcal{C}_n}^{(s+1)}(\mathbf{d})}{g^{s+1}}+\mathcal{O}_{n,s}\left(\frac{|\mathbf{d}|^{2s}}{g^{s+2}}\right)
\end{align*}

where: 

\begin{align*}
    e_{\delta_1\mathcal{C}_n}^{(s+1)}(\mathbf{d})=\sum_{\genfrac{}{}{0pt}{}{g_1+g_2=g}{I\amalg J=\{2,\ldots, n\}}}&\sum_{r=0}^{s+1-2g_1-|I|}\mathfrak{c}_{n,g_1,|I|}^{(s+1-r)}\cdot e_{\delta\mathcal{C},g_1,J}^{(r)}. 
\end{align*}

Rearranging polynomials and partial polynomials completes the proof of Claim~\ref{claim:C}.

\end{proof}

\subsection{Discrete integration}
Discrete integration allows us to derive the expansion of intersection numbers from the expansion of their derivative. The symmetry of intersection numbers which we recalled in Proposition~\ref{prop:internbMZ} makes the derivative  along the first variable sufficient. 

\begin{align*}
    \frac{[\tau_{d_{1}}\,\ldots \tau_{d_n}]_{g,n}}{V_{g,n}}-1&= \sum_{i=0}^{d_n-1}\frac{[\tau_{i+1},\tau_0\ldots \tau_0]_{g,n}-[\tau_{i},\tau_0,\,\ldots \tau_{0}]_{g,n}}{V_{g,n}}\\
    &+\sum_{i=0}^{d_{n-1}-1}\frac{[\tau_{i+1},\tau_0,\ldots \tau_0,\tau_{d_n}]_{g,n}-[\tau_i,\tau_0\ldots \tau_{d_n}]_{g,n}}{V_{g,n}}\\
    &+\ldots \\
    &+\sum_{i=0}^{d_1-1}\frac{[\tau_{i+1},\tau_{d_2},\ldots \tau_{d_{n}}]_{g,n}-[\tau_i,\tau_{d_2},\ldots \tau_{d_{n}}]_{g,n}}{V_{g,n}}
\end{align*}
We decompose discrete derivatives in the three terms of topological recursion~\ref{eq:discrderivative}. 
\begin{align*}
    \frac{[\tau_{d_{1}}\,\ldots \tau_{d_n}]_{g,n}}{V_{g,n}}-1&=\sum_{i=0}^{d_n-1}\left(\frac{\delta_1\mathcal{A}_n}{V_{g,n}}+\frac{\delta_1\mathcal{B}_n}{V_{g,n}}+\frac{\delta_1\mathcal{C}_n}{V_{g,n}}\right)(i,0,\ldots 0)\\
    &+\sum_{i=0}^{d_{n-1}-1}\left(\frac{\delta_1\mathcal{A}_n}{V_{g,n}}+\frac{\delta_1\mathcal{B}_n}{V_{g,n}}+\frac{\delta_1\mathcal{C}_n}{V_{g,n}}\right)(i,0,\ldots 0, d_n)\\
    &+\ldots\\
    &+\sum_{i=0}^{d_1-1}\left(\frac{\delta_1\mathcal{A}_n}{V_{g,n}}+\frac{\delta_1\mathcal{B}_n}{V_{g,n}}+\frac{\delta_1\mathcal{C}_n}{V_{g,n}}\right)(i,d_2,\ldots d_{n})\\
\end{align*}

We substitute to $\frac{\delta_1\mathcal{A}_n}{V_{g,n}}$, $\frac{\delta_1\mathcal{B}_n}{V_{g,n}}$ and $\frac{\delta_1\mathcal{C}_n}{V_{g,n}}$ their $s+1$-th order asymptotic expansion. By rearranging sums, one obtains

\begin{align*}
    e_{n}^{(s+1)}(\mathbf{d})
    &=\sum_{i=0}^{d_n-1}\left(e_{\delta_1{\mathcal{A}}_{n}/{V_{n}}}^{(s+1)}+e_{\delta_1\mathcal{B}_n/V_{n}}^{(s+1)}+e_{\delta_1\mathcal{C}_n}^{(s+1)}\right)(i,0,\ldots 0)\\
    &+\sum_{i=0}^{d_{n-1}-1}\left(e_{\delta_1{\mathcal{A}}_{n}/{V_{n}}}^{(s+1)}+e_{\delta_1\mathcal{B}_n/V_{n}}^{(s+1)}+e_{\delta_1\mathcal{C}_n}^{(s+1)}\right)(i,0,\ldots 0, d_n)\\
    &+\ldots\\
    &+\sum_{i=0}^{d_1-1}\left(e_{\delta_1{\mathcal{A}}_{n}/{V_{n}}}^{(s+1)}+e_{\delta_1\mathcal{B}_n/V_{n}}^{(s+1)}+e_{\delta_1\mathcal{C}_n}^{(s+1)}\right)(i,d_2,\ldots d_{n}).
\end{align*}

\begin{claim}\label{claim:s+1}
The $(s+1)$-th coefficient of this expansion admits a decomposition:
\begin{align}
    e_{n}^{(s+1)}(\mathbf{d})
    &=q_{n}^{(s+1)}(\mathbf{d})+\sum_{J\subsetneq\{\mathbf{n}\}} q_{n,J}^{(s+1)}(\mathbf{d})\label{eq:decompo}\\
    &=\left(\frac{-1}{\pi^{2}}\right)^{s+1}\sum_{\substack{i_1+\ldots+i_n\\=2s+2}}\frac{(2s+1)!!}{i_1!\ldots i_n!}d_1^{i_1}\ldots d_n^{i_n}+\mathcal{O}_{n,s}\left(|\mathbf{d}|^{2s+1}\right)\label{eq:explicit}
\end{align}
where $q_{n}^{(s+1)}$ is a polynomial in $n$ variables of degree $2s+2$ and $q_{n,J}^{(s+1)}$ are $J$-partial polynomials of degree $\mathrm{deg}(q_{n,J}^{(s+1)})\leq 2s-2$ and $\mathrm{supp}(q_{n,J}^{(s+1)})\subset\{|\mathbf{d}_{\{\mathbf{n}\}\setminus J}|<2s-2\}$.
\end{claim}

\begin{proof}[Proof of Claim~\ref{claim:s+1}]

Through discrete integration, the degree of the polynomials and partial polynomials increases by one by Faulhaber's formula, see Remark~\ref{rem:convo}. For partial polynomials, it behaves slightly differently, see Remark~\ref{rem:partpol}. Thus we have the decomposition~\ref{eq:decompo} in polynomial and partial polynomials, with degrees and support as announced.

For the computation of the dominant terms, let us look at the last sum of the discrete integration. The former ones do not depend on $d_1$: the polynomials obtained might be of degree $2s+2$ but only in $(d_2,\ldots d_n)$. As we know that intersection numbers are symmetric, it is is sufficient to look at the last sum to compute the whole polynomial.

\begin{align*}
    q_{n}^{(s+1)}(\mathbf{d})&=\sum_{i=0}^{d_1-1}\left(e_{\delta_1{\mathcal{A}}_{n}/{V_{n}}}^{(s+1)}+e_{\delta_1\mathcal{B}_n/V_{n}}^{(s+1)}+e_{\delta_1\mathcal{C}_n}^{(s+1)}\right)(i,d_2,\ldots d_{n})\\
    &=\sum_{i=0}^{d_1-1}\left(\frac{-1}{\pi^{2}}\right)^{s+1}\sum_{\substack{i_1+\ldots\\+i_n=2s}}\sum_{j=2}^{n}\frac{(2s-1)!!}{i_1!\ldots i_n!}\cdot d_1^{i_1}\ldots d_j^{i_j+1} \ldots d_n^{i_n}\\
    &\hspace{3.5cm}+\sum_{i=0}^{d_1-1}\left(\frac{-1}{\pi^{2}}\right)^{s+1}\sum_{\substack{i_1+\ldots\\+i_n=2s}}\frac{(2s-1)!!}{i_1!\ldots i_n!}\cdot d_1^{i_1+1} d_2^{i_2}\ldots d_n^{i_n}\\
    &\hspace{4.8cm}+\mathcal{O}_{n,s}\left((d_2+\ldots+ d_n)^{2s+2}+|\mathbf{d}|^{2s+1}\right)
\end{align*}

By Remark~\ref{rem:convo},

\begin{align*}
    q_{n}^{(s+1)}(\mathbf{d})&=\left(\frac{-1}{\pi^{2}}\right)^{s+1}\sum_{\substack{i_1+\ldots\\+i_n=2s}}\sum_{j=2}^{n}\frac{(2s-1)!!}{(i_1+1)!\ldots i_n!}\cdot d_1^{i_1+1}\ldots d_j^{i_j+1} \ldots d_n^{i_n}\\
    &\hspace{3.1cm}+\left(\frac{-1}{\pi^{2}}\right)^{s+1}\sum_{\substack{i_1+\ldots\\+i_n=2s}}\frac{(2s-1)!!(i_1+1)}{(i_1+2)!\ldots i_n!}\cdot d_1^{i_1+2} d_2^{i_2}\ldots d_n^{i_n}\\
    &\hspace{4.8cm}+\mathcal{O}_{n,s}\left((d_2+\ldots+ d_n)^{2s+2}+|\mathbf{d}|^{2s+1}\right)\\
    &=\left(\frac{-1}{\pi^{2}}\right)^{s+1}\sum_{\substack{i_1+\ldots+i_n=2s+2\\i_1\geq 1}}\frac{(2s-1)!!}{i_1!\ldots i_n!}d_1^{i_1}\ldots d_n^{i_n}\left(\sum_{j=2}^{n}i_j + i_1-1\right)\\
    &\hspace{4.8cm}+\mathcal{O}_{n,s}\left((d_2+\ldots +d_n)^{2s+2}+|\mathbf{d}|^{2s+1}\right)\\
    &=\left(\frac{-1}{\pi^{2}}\right)^{s+1}\sum_{\substack{i_1+\ldots+i_n=2s+2\\i_1\geq 1}}\frac{(2s+1)!!}{i_1!\ldots i_n!}\cdot d_1^{i_1}\ldots d_n^{i_n}\\
    &\hspace{4.8cm}+\mathcal{O}_{n,s}\left((d_2+\ldots+ d_n)^{2s+2}+|\mathbf{d}|^{2s+1}\right)
\end{align*}

By symmetry, 

\begin{align*}
    q_{n}^{(s+1)}(\mathbf{d})=\left(\frac{-1}{\pi^{2}}\right)^{s+1}\sum_{\substack{i_1+\ldots+i_n\\=2s+2}}\frac{(2s+1)!!}{i_1!\ldots i_n!}\cdot d_1^{i_1}\ldots d_n^{i_n}+\mathcal{O}_{n,s}\left(|\mathbf{d}|^{2s+1}\right)
\end{align*}

and 

\begin{align*}
    e_{n}^{(s+1)}(\mathbf{d})=\left(\frac{-1}{\pi^{2}}\right)^{s+1}\sum_{\substack{i_1+\ldots+i_n\\=2s+2}}\frac{(2s+1)!!}{i_1!\ldots i_n!}\cdot d_1^{i_1}\ldots d_n^{i_n}+\mathcal{O}_{n,s}\left(|\mathbf{d}|^{2s+1}\right).
\end{align*}
The proof of Claim~\ref{claim:s+1} is complete.
\end{proof}

\subsection{Bound on the coefficients of the expansion}\label{subsec:bound}
It remains to prove that the other coefficients of the polynomials and partial polynomials are bounded by $(C_ns)^{s}$.
\begin{claim}\label{claim:controlcoef}
There exist a constant $C_n$ depending on $n$ such that for any $s,\mathbf{d}$: $e_n^{(s)}(\mathbf{d})$ is bounded by: 
\begin{align*}
    |e_{n}^{(s)}(\mathbf{d})|\leq \sum_{\substack{i_1+\ldots \\+i_n\leq 2s}}(C_ns)^{s}\frac{d_1^{i_1}\ldots d_n^{i_n}}{i_1!\ldots i_n!}. 
\end{align*}
\end{claim}

In \cite{HMT}, the authors explain that the control over the polynomial coefficients of the expansion of intersection numbers require the study of three other expansions: the expansion of the two volume ratios of Theorem~\ref{thm:expratios} and the expansion of the discrete derivative of intersection numbers. Their proof goes through an induction in 3 steps, see the schematic of the proof at the beginning of Section 4.3 in \cite{HMT}: 
\begin{enumerate}
    \item The expansion of intersection numbers and of volume ratios up to order $s$ yields the expansion of the discrete derivative up to order $s+1$. 
    \item The expansion of the discrete derivative up to order $s+1$ yields the expansion of intersection numbers up to order $s+1$. 
    \item The expansion of intersection numbers up to order $s+1$ yields the expansion of the volume ratios up to order $s+1$. 
\end{enumerate}

Our approach differs mainly concerning the first step, therefore we refer to their paper for the latter two. The main ideas which allow us to improve the control of the coefficients are already present in our explicit computation of the main term. We don't bound separately the three terms of topological recursion, as it is done in Lemma 4.12, 4.13 and 4.14 in \cite{HMT}. This allow us to get only one factor $s$, which is analogue to the factor $(2s+1)$ which appears in our main term computation: 
$$2s+1=\sum_{j=2}^{n}i_j+i_1+1$$
where the sum comes from the first term of the recursion, and $i_1+1$ comes from the second term. 

\begin{proof}[Proof of Claim~\ref{claim:controlcoef}]
We will first show that, for any $\mathbf{d}$, the main polynomials verify: 
\begin{equation*}
    |q_n^{(s)}(\mathbf{d})|\leq \sum_{\substack{i_1+\ldots\\+i_n\leq 2s}}\frac{(C_ns)^{s}}{i_1!\ldots i_n!}d_1^{i_1}\ldots d_n^{i_n}.
\end{equation*}

Only the two first terms of the topological recursion need to be considered in order to compute main polynomials. For the derivative of the first term, we get the following bound using the induction hypothesis: 
\begin{align*}
    |q_{\delta_1\mathcal{A}_n/V_{n-1}}^{(s)}(\mathbf{d})|&=|8\;\sum_{j=2}^{n}\sum_{i=0}^{\infty} (\!a_{i-1}\!-\!a_{i}\!)(2d_j+1) \cdot q_{n-1}^{(s)}(d_1\!+\!d_j\!+\!i\!-\!1,d_2\ldots\widehat{d_j},\ldots d_n)|\\
    &\leq 8\;\sum_{j=2}^{n}\sum_{i=-1}^{\infty} (\!a_{i+1}\!-\!a_{i}\!)(2d_j+1) \cdot |q_{n-1}^{(s)}(d_1\!+\!d_j\!+\!i\!,d_2\ldots\widehat{d_j},\ldots d_n)|\\
    &\leq 8\;\sum_{j=2}^{n}(2d_j+1)\cdot\sum_{i=-1}^{\infty} (\!a_{i+1}\!-\!a_{i}\!)\cdot(C_ns)^{s}\\
    &\hspace{4.6cm}\cdot\sum_{\substack{i_1'+\ldots+\hat{i_j}\\+\ldots+i_n\leq 2s}}\frac{(d_1+d_j+i)^{i_1'}\ldots d_n^{i_n}}{i_1'!\ldots\hat{i_j} \ldots i_n!}
\end{align*}

We need to bound
\begin{align*}
    \sum_{i=-1}^{\infty} (\!a_{i+1}\!-\!a_{i}\!)(d_1+d_j+i)^{i_1'}=\sum_{i_1+i_j+i_0=i_1'}i_1'!\frac{d_1^{i_1}d_j^{i_j}}{i_1!i_j!}\cdot \frac{1}{i_0!}\sum_{i=-1}^{\infty} (\!a_{i+1}\!-\!a_{i}\!)i^{i_0}.
\end{align*}
\end{proof}

Therefore we use Lemma 4.5 in \cite{HMT}:
\begin{lem}[Hide-Macera-Thomas]\label{lem:HMT1}
For any $r\in\mathbb{N}$, 
\begin{equation*}
    \sum_{i=-1}^{\infty}(\!a_{i+1}\!-\!a_{i}\!)i^{r}\leq 2r!
\end{equation*}
\end{lem}

Thus we obtain: 

\begin{align*}
    |q_{\delta_1\mathcal{A}_n/V_{n-1}}^{(s)}(\mathbf{d})|\leq& 32\cdot(C_ns)^{s}\sum_{\substack{i_1+\ldots\\+i_n\leq 2s+1}}\cdot\frac{d_1^{i_1}\ldots d_n^{i_n}}{i_1!\ldots i_n!}\sum_{j=2}^{n}i_j\\
    &+16\cdot(n-1)(C_ns)^{s}\sum_{\substack{i_1+\ldots\\+i_n\leq 2s}}\frac{d_1^{i_1}\ldots d_n^{i_n}}{i_1!\ldots i_n!}.
\end{align*}

For the derivative of the second term, by induction hypothesis we have: 
\begin{align*}
    |q_{\delta_1\mathcal{B}_n/V_{n+1}}^{(s)}(\mathbf{d})|&=|16\sum_{i=0}^{\infty}\sum_{\substack{k_1+k_2=\\d_1+i-2}}(a_{i-1}-a_{i}) \cdot q_{n+1}^{(s)}(k_1,k_2,d_2\ldots d_n)|\\
    &\leq 16\sum_{i=-1}^{\infty}\sum_{\substack{k_1+k_2=\\d_1+i-1}}(a_{i+1}-a_{i}) \cdot (C_ns)^{s}\sum_{\substack{i_{k_1}+i_{k_2}+i_2\\+\ldots+i_n\leq 2s}}\frac{k_1^{i_{k_1}}k_2^{i_{k_2}}d_2^{i_2}\ldots d_n^{i_n}}{i_{k_1}!i_{k_2}!i_2!\ldots i_n!}\\
    &\leq 16\cdot(C_ns)^{s}\sum_{\substack{i_1'+i_2+\ldots\\+i_n\leq 2s}}\frac{d_2^{i_2}\ldots d_n^{i_n}}{i_1!i_2!\ldots i_n!}\sum_{i=-1}^{\infty}(a_{i+1}-a_{i})(d_1+i-2)^{i_1'+1}.
\end{align*}

As for the first term, one obtains through Lemma~\ref{lem:HMT1}:

\begin{align*}
    \sum_{i=-1}^{\infty}(a_{i+1}-a_{i})(d_1+i-2)^{i_1'+1}&=\sum_{i_0+i_1=i_1'+1}\frac{(i_1'+1)!}{i_0!i_1!}d_1^{i_1}\sum_{i=-1}^{\infty}(a_{i+1}-a_{i})(i-1)^{i_0}\\
    &\leq 2\sum_{i_0+i_1=i_1'+1}\frac{(i_1'+1)!}{i_1!}d_1^{i_1}
\end{align*}

As $i_1'\leq 2s-\sum_{j=2}^{n}i_j$, the main polynomial of the derivative of the second term is bounded by: 

\begin{align*}
    |q_{\delta_1\mathcal{B}_n/V_{n+1}}^{(s)}(\mathbf{d})|\leq  32\cdot (C_ns)^{s}\sum_{\substack{i_1+\ldots\\+i_n\leq 2s+1\\\sum_{2}^n i_j\leq 2s}}\left(2s+1-\sum_{j=2}^{n}i_j\right)\frac{d_1^{i_1}\ldots d_n^{i_n}}{i_1!\ldots i_n!}. 
\end{align*}

Thanks to discrete integration, one can deduce a bound on the coefficients for $i_1\geq 1$, which corresponds to: 

\begin{align*}
    |q_n^{(s+1)}(\mathbf{d})-q_n^{(s+1)}(0,d_2\ldots d_n)|&\leq 32(C_ns)^{s}\sum_{\substack{i_1+\ldots\\+i_n\leq 2s+2}}\frac{d_1^{i_1}\ldots d_n^{i_n}}{i_1!\ldots i_n!}\sum_{j=2}^{n}i_j\\
    &\quad+16(n-1)(C_ns)^{s}\sum_{\substack{i_1+\ldots\\+i_n\leq 2s+1\\i_1\geq 1}}\frac{d_1^{i_1}\ldots d_n^{i_n}}{i_1!\ldots i_n!}\\
    +32\cdot (C_ns)^{s}&\sum_{\substack{i_1+\ldots\\+i_n\leq 2s+2\\i_1\geq 1;\sum_{2}^n i_j\leq 2s}}\left(2s+1-\sum_{j=2}^{n}i_j\right)\frac{d_1^{i_1}\ldots d_n^{i_n}}{i_1!\ldots i_n!}\\
    &\leq (C_ns)^{s}\cdot 32\cdot(2s+1)\sum_{\substack{i_1+\ldots\\+i_n\leq 2s+2}}\frac{d_1^{i_1}\ldots d_n^{i_n}}{i_1!\ldots i_n!}\\
    &\leq (C_n(s+1))^{s+1}\sum_{\substack{i_1+\ldots\\+i_n\leq 2s+2}}\frac{d_1^{i_1}\ldots d_n^{i_n}}{i_1!\ldots i_n!}.
\end{align*}
The symmetry of the polynomials concludes the induction for main polynomials.

For partial polynomials, one gets a similar bound on the coefficients: 
\begin{align*}
    |q_{n,J}^{(s)}(\mathbf{d})|&\leq (C_ns)^{s}\sum_{\substack{i_{j_1}+\ldots\\+i_{j_{|J|}}\leq 2s-4}}\frac{d_{j_1}^{i_{j_1}}\ldots d_{j_{|J|}}^{i_{j_{|J|}}}}{i_{j_1}!\ldots i_{j_{|J|}}!} 
\end{align*}

Partial polynomials appear in the third term of the topological recursion. Let us first show:
\begin{align*}
     |q_{\delta\mathcal{C},g_1,J}^{(s-1)}|\leq (C_ns)^{s}\sum_{\substack{i_1+i_{j_1}+\ldots i_{j_{|J|}}}}\frac{d_1^{i_1}d_{j_1}^{i_{j_1}}\ldots d_{j_{|J|}}^{i_{j_{|J|}}}}{i_1!i_{j_1}!\ldots i_{j_{|J|}}!}
\end{align*}

We recall that, as in the proof of Claim~\ref{claim:C1}: 
\begin{align*}
    q_{\delta\mathcal{C},g_1,J}^{(s-1)}(\mathbf{d}_J)&=\sum_{i=0}^{\infty} \sum_{\genfrac{}{}{0pt}{}{k_{1}+k_{2}}{=i+d_{1}-2}}(a_{i-1}-a_{i})&
    \cdot\left[\tau_{k_{1}} \prod_{i\in I } \tau_{d_{i}}\right]_{g_1,|I|+1}\cdot q_{|J|+1}^{(s+1-2g_1-|I|)}(k_2,\mathbf{d}_J)
\end{align*}
Then by induction hypothesis: 
\begin{align*}
    |q_{\delta\mathcal{C},g_1,J}^{(s-1)}(\mathbf{d}_J)|&\leq\sum_{i=-1}^{\infty} \sum_{\genfrac{}{}{0pt}{}{k_{1}+k_{2}}{=i+d_{1}-1}}(a_{i+1}-a_{i})\left[\tau_{k_{1}} \prod_{i\in I } \tau_{d_{i}}\right]_{g_1,|I|+1}
    \\
    &\cdot(C_n(s+1-2g_1-|I|))^{s+1-2g_1-|I|}\sum_{\substack{i_0+i_{j_1}+\ldots i_{j_{|J|}}\\\leq 2(s+1-2g_1-|I|)}}\frac{k_2^{i_0}d_{j_1}^{i_{j_1}}\ldots d_{j_{|J|}}^{i_{j_{|J|}}}}{i_0!i_{j_1}!\ldots i_{j_{|J|}}!}
\end{align*}

To control $\left[\tau_{k_{1}} \prod_{i\in I } \tau_{d_{i}}\right]_{g_1,|I|+1}$, we apply lemma 2.6 in \cite{HMT}, that is: 

\begin{lem}[Hide-Macera-Thomas]\label{lem:HMT2}
\begin{align*}
    V_{g,n}\leq C^{2g+n}(2g+n)!
\end{align*}
\end{lem}

Thus: 

\begin{align*}
    \left[\tau_{k_{1}} \prod_{i\in I } \tau_{d_{i}}\right]_{g_1,|I|+1}&\leq V_{g_1,|I|+1}\\
    &\leq (C(2g+|I|))^{2g_1+|I|-2}.
\end{align*}

Moreover, summing over $k_1+k_2=i+d_1-1$ yields: 

\begin{align*}
    \sum_{\genfrac{}{}{0pt}{}{k_{1}+k_{2}}{=i+d_{1}-1}}k_2^{i_0}&=\frac{(i+d_1-1)^{i_0+1}}{i_0+1}\\
    &=\frac{1}{i_0+1}\sum_{\substack{i_1+i_0'\\=i_0+1}}\frac{d_1^{i_1}(i-1)^{i_0'}}{i_1!i_0'!}.
\end{align*}

Applying the inequality of Lemma~\ref{lem:HMT1} yields us:

\begin{align*}
    |q_{\delta\mathcal{C},g_1,J}^{(s-1)}(\mathbf{d}_J)|\leq(2(s+1-2g_1-|I|))\cdot (C_ns)^{s-1}\sum_{\substack{i_1+i_{j_1}+\ldots i_{j_{|J|}}\\\leq 2(s+1-2g_1-|I|)}}\frac{d_1^{i_1}d_{j_1}^{i_{j_1}}\ldots d_{j_{|J|}}^{i_{j_{|J|}}}}{i_1!i_{j_1}!\ldots i_{j_{|J|}}!}\\
    \leq (C_ns)^{s}\sum_{\substack{i_1+i_{j_1}+\ldots i_{j_{|J|}}\\\leq 2(s+1-2g_1-|I|)}}\frac{d_1^{i_1}d_{j_1}^{i_{j_1}}\ldots d_{j_{|J|}}^{i_{j_{|J|}}}}{i_1!i_{j_1}!\ldots i_{j_{|J|}}!}.
\end{align*}

To compute the total contribution of $\delta_1{\mathcal{C}}_{g,n}(\mathbf{d})/V_{g,n}$, we sum over $2g_1\leq s-|I|-1$. Through Faulhaber's formula, we obtain the bound $(C(s+1))^{s+1}$. This bound for partial polynomials is well preserved through the induction. The proof is analogous as the one for the main polynomials. 

The sum of partial polynomials is bounded by: 
\begin{align*}
    \sum_{J\in\{\mathbf{n}\}}q_{n,J}^{(s+1)}(\mathbf{d})\leq 2^{n}(C_n(s+1))^{s+1}\sum_{\substack{i_1+\ldots\\+i_n\leq 2s+2}}\frac{d_1^{i_1}\ldots d_n^{i_n}}{i_1!\ldots i_n!}.
\end{align*}

The proof of Theorem~\ref{thm:internb} is complete.

\end{proof}
\section{Expansion of volume polynomials}\label{sec:expvol}

We are now able to refine the expansion of volumes of moduli spaces. We mostly follow the technique of Anantharaman and Monk \cite{AM22}. 

\subsection{Degree of the polynomials in the volumes expansion}
\begin{proof}[Proof of Theorem~\ref{thm:expvol}]
We deduce the asymptotic expansion of the volume polynomials from the expansion of the coefficients $\left[\tau_{d_1}\ldots\tau_{d_n}\right]$. Let us recall the following Lemma from \cite{AM22}:
\begin{lem}
  \label{lemm:sum}
  Let $k \geq 0$ be an integer and 
  \begin{equation*}
    \mathrm{p}_k(X) =  \prod_{j=0}^{k-1} (2X+1-j)
    = (2X+1) (2X) (2X-1) \ldots (2X+2-k),
  \end{equation*}. 
  This polynomials are a basis of the set of polynomials and for any $x \in \mathbb{R}$,
  \begin{equation*}
    \sum_{d=0}^{+ \infty} \frac{\mathrm{p}_{k}(d) \, x^{2 d+1}}{2^{2 d+1}(2 d+1)!}
    =
    \begin{cases}
      \frac{x^{k}}{2^{k}} \sinh (\frac{x}{2}) & \text{if } k \text{ is even} \\
      \frac{x^{k}}{2^{k}} \cosh (\frac{x}{2}) & \text{if } k \text{ is odd}.
    \end{cases}.
  \end{equation*}
\end{lem}

To compute the expansion of: 
\begin{align*}
\prod_{j=1}^{n}\frac{x_j}{2}\frac{V_{g,n}(\mathbf{x})}{V_{g,n}}
&= \sum_{d_1, \ldots, d_n} \frac{\left[\tau_{d_1}\ldots\tau_{d_n}\right]}{V_{g,n}}
\prod_{j=1}^n \frac{x_j^{2d_j+1}}{2^{2d_j+1}(2d_j+1)!}\\
&= \sum_{d_1, \ldots,d_n}\left(\sum_{r=0}^{s}\frac{e_n^{(r)}(\mathbf{d})}{g^r} +\mathcal{O}_{n,s}\left(\frac{|\mathbf{d}|^{s+1}}{g^{s+1}}\right)\right)
\prod_{j=1}^n \frac{x_j^{2d_j+1}}{2^{2d_j+1}(2d_j+1)!}\\
\end{align*}

The coefficient $e_n^{(r)}$ is decomposed in polynomials $q_n^{(r)}$ and $J$-partial polynomials $q_{n,J}^{(r)}$. The polynomial $q_n^{(r)}$ can be written as sum of products of $p_k(d_i)$ with $0\leq k \leq 2r$. Then we obtain the following contribution:
   \begin{equation*}
    \sum_{d_1, \ldots,d_n}\!q_n^{(r)}\!(\mathbf{d})\!
\prod_{j=1}^n\! \frac{x_j^{2d_j+1}}{2^{2d_j+1}(2d_j+1)!}
    \!=\! \sum_{\substack{J_+ \sqcup J_-\\=\{1, \ldots,n \}}}
    P_{g,n}^{(r,J_\pm)}(\mathbf{x}) \prod_{i \in J_+} \cosh
    \frac{x_i}{2}
    \prod_{i \in J_-} \sinh \frac{x_i}{2}.
  \end{equation*}
  where $P_{g,n}^{(r,J_\pm)}$ is a polynomial of degree $2r$. 
  
For partial polynomials, we obtain: 

\begin{equation*}
    \sum_{d_1, \ldots,d_n}\!q_{n,J}^{(r)}\!(\mathbf{d})\!
\prod_{j=1}^n \frac{x_j^{2d_j+1}}{2^{2d_j+1}(2d_j+1)!}
    \!=\! \sum_{J_+ \sqcup J_=J}
    P_{g,n}^{(r,J_\pm)}(\mathbf{x}) \prod_{i \in J_+} \cosh
    \frac{x_i}{2}
    \prod_{i \in J_-} \sinh \frac{x_i}{2}.
  \end{equation*}
where $P_{g,n}^{(r,J_\pm)}$ is a polynomial in $\mathbf{x}_{J}$ of degree at most $2r-4$ and polynomial in $\mathbf{x}_{\{\mathbf{n}\}\setminus J}$ of degree strictly less than $2r-4$. 
  
For the error term: 
\begin{equation*}
    \sum_{d_1, \ldots,d_n}\mathcal{O}_{n,s}(|\mathbf{d}|^{2s+2})\prod_{j=1}^n \frac{x_j^{2d_j+1}}{2^{2d_j+1}(2d_j+1)!}=\mathcal{O}_{n,s}\left(\frac{|\mathbf{x}|^{2s+2}}{g^{s+1}}\cdot\exp\frac{|\mathbf{x}|}{2}\right).
\end{equation*}

\subsection{Bound on the polynomials}

It remains to prove that:
\begin{align*}
      |Q_n^{(r,J_\pm)}(x_1,\ldots x_n)|\leq(C_nr)^{r}(1+|\mathbf{x}|)^{2r}. 
\end{align*}

The bound we obtain on the coefficients of the expansion of intersection numbers directly yields this result. The proof can be found in Section 4.4 of \cite{HMT}. We highlight that the choice of the basis $\frac{d_1^{i_1}\ldots d_n^{i_n}}{i_1!\ldots i_n!}$ in the decomposition of intersection numbers allows to go through the change of basis of Lemma~\ref{lemm:sum} introduced in \cite{AM22} and get the analogous bound on the coefficients of the expansion of volumes. 
\end{proof}

\subsection{Exact asymptotic value at order $\mathcal{O}(\sqrt{g})$} The exact computation of the dominant terms in Theorem~\ref{thm:internb} yields us Theorem~\ref{thm:asymsqrt}.
\begin{proof}[Proof of Theorem~\ref{thm:asymsqrt}]
    
In the basis $\{p_k\}_{k}$, the main polynomials are expressed as follow:
\begin{align*}
q_n^{(r)}(\mathbf{d})&=\left(\frac{-1}{\pi^{2}}\right)^{r+1}\sum_{\substack{i_1+\ldots+i_n\\=2r+2}}\frac{(2r+1)!!}{i_1!\ldots i_n!}\cdot d_1^{i_1}\ldots d_n^{i_n}
    +\mathcal{O}_{n,r}\left(|\mathbf{d}|^{2r+1}\right)\\
    &=\left(\frac{-1}{\pi^{2}}\right)^{r}\sum_{\substack{i_1+\ldots+i_n\\=2r}}\frac{(2r-1)!!}{i_1!\ldots i_n!}\frac{1}{2^{2r}}\cdot p_{i_1}(d_1)\ldots p_{i_n}(d_n)+\mathcal{O}_{n,r}\left(|\mathbf{d}|^{2r-1}\right)
\end{align*}

From Lemma~\ref{lemm:sum}:

\begin{align*}
(\star):&=\sum_{d_1, \ldots,d_n}q_n^{(r)}(\mathbf{d})
\prod_{j=1}^n \frac{x_j^{2d_j+1}}{2^{2d_j+1}(2d_j+1)!}\\
&=\left(\left(\frac{-1}{\pi^{2}}\right)^{r}\sum_{\substack{i_1+\ldots+i_n\\=2r}}\frac{(2r-1)!!}{i_1!\ldots i_n!}\frac{1}{2^{4r}}x_1^{i_1}\ldots x_n^{i_n}\cdot\prod_{i_j \text{ even}}\sinh\frac{x_j}{2}\prod_{i_j \text{ odd}}\cosh\frac{x_j}{2}\right)\\
&\hspace{8.5cm}\cdot\left(1+\mathcal{O}_{n}\left(\frac{1}{\min_{j}x_j}\right)\right)
\end{align*}

As $x_i\rightarrow \infty$ and $x_i=\mathcal{O}(\sqrt{g})$:
$$\prod_{i_j \text{ even}}\sinh\frac{x_j}{2}\prod_{i_j \text{ odd}}\cosh\frac{x_j}{2}=\frac{1}{2^{n}}\exp\frac{\mathbf{|x|}}{2}\left(1+\mathcal{O}_{n}\left(\frac{1}{\exp\min\limits_{1\leq i \leq n} x_i}\right)\right).$$
    
Then: 

\begin{align*}
(\star)&=\frac{1}{2^{n}}\left(\left(\frac{-1}{\pi^{2}}\right)^{r}\sum_{\substack{i_1+\ldots+i_n\\=2r}}\frac{(2r-1)!!}{i_1!\ldots i_n!}\frac{1}{2^{4r}}x_1^{i_1}\ldots x_n^{i_n}\cdot\exp\frac{|\mathbf{x}|}{2}\right)\left(1+\mathcal{O}_{n}\left(\frac{1}{\min_{j}x_j}\right)\right)\\
&=\frac{1}{2^{n}}\frac{1}{r!}\left(\frac{-1}{16\pi^{2}}\right)^{r}\frac{1}{2^{r}}\left(\sum_{\substack{i_1+\ldots+\\i_n=2r}}\frac{(2r)!}{i_1!\ldots i_n!}x_1^{i_1}\ldots x_n^{i_n}\right)\exp\frac{\mathbf{|x|}}{2}\left(1+\mathcal{O}_{n}\left(\frac{1}{\min_{j}x_j}\right)\right)\\
&=\frac{1}{2^{n}}\frac{1}{r!}\left(\frac{-1}{32\pi^{2}}\right)^{r}|\mathbf{x}|^{2r}\exp\frac{\mathbf{x}}{2}\left(1+\mathcal{O}_{n}\left(\frac{1}{\min_{j}x_j}\right)\right).
\end{align*}

Dominant terms are summable over $r$: 
\begin{align*}
    \left(\sum_{r=0}^{\infty}\frac{1}{r!}\left(\frac{-1}{32\pi^{2}}\right)^{r}|\mathbf{x}|^{2r}\right)\exp\frac{|\mathbf{x}|}{2}=\exp\left(\frac{|\mathbf{x}|}{2}-\frac{|\mathbf{x}|^{2}}{32\pi^{2}g}\right).
\end{align*}

The proof of Theorem~\ref{thm:asymsqrt} when $x_j(g)\rightarrow\infty$ for all $j\in \{\mathbf{n}\}$ is complete. 

Let us denote $J$ the subset of $\{\mathbf{n}\}$ such that for $j\in J $, $x_j(g)\rightarrow\infty$. Then: 

\begin{align*}
(\star)&=\frac{1}{2^{|J|}}\left(\frac{-1}{\pi^{2}}\right)^{r}\sum_{\substack{i_{j_1}+\ldots+i_{j_{|J|}}\\=2r}}\frac{(2r-1)!!}{i_{j_1}!\ldots i_{j_{|J|}}!}\frac{1}{2^{4r}}x_1^{i_{j_1}}\ldots x_n^{i_{j_{|J|}}}\cdot\exp\frac{|\mathbf{x}_J|}{2}\\
&\hspace{6cm}\cdot\prod_{j\notin J}\sinh\frac{x_j}{2}\cdot\left(1+\mathcal{O}_{n,|J|}\left(\frac{1}{\min_{j\in J}x_j}\right)\right)\\
\end{align*}

As previously, one can sum over $r$: 
\begin{align*}
    \sum_{r=0}^{\infty}\left(\frac{-1}{\pi^{2}}\right)^{r}\sum_{\substack{i_{j_1}+\ldots+\\i_{j_{|J|}}=2r}}\frac{(2r-1)!!}{i_{j_1}!\ldots i_{j_{|J|}}!}\frac{1}{2^{4r}}x_1^{i_{j_1}}\ldots x_n^{i_{j_{|J|}}}=\exp\frac{-|\mathbf{x}_J|^{2}}{32\pi^{2}g}. 
\end{align*}

And the expansion of the volume is: 

\begin{align*}
    \prod_{j=1}^{n}\frac{x_j}{2}\cdot\frac{V_{g,n}(x_1,\ldots x_n)}{V_{g,n}}\!=\!\frac{1}{2^n}\exp\left({\frac{|\mathbf{x}|}{2}\!-\!\frac{1}{8\pi^{2}g}\left(\frac{|\mathbf{x}|}{2}\right)^{2}}\right)\!\left(1\!+\!\mathcal{O}_{n}\left(\frac{1}{\min x_j}\right)\right).
\end{align*}

\end{proof}

\section{Counting geodesics of length at most $L=O(\sqrt{g})$}\label{sec:countres}

The refinement of the expansion of volume polynomials allows us to understand the behavior of the counting function of simple closed geodesics at the threshold $\sqrt{g}$. In fact, one can deduce expectations of geometrical functions from the asymptotic expansion of volumes through Mirzakhani's integration formula. We recall the statement of Theorem~\ref{thm:count}:

\begin{thm*}
For $L=\mathcal{O}(\sqrt{g})$, $L\rightarrow \infty$
\begin{align}
    \mathbb{E}_{WP}(N^{s}_{nsep}(X,L))\sim_g\frac{e^{L-\frac{L^{2}}{8\pi^{2}g}}}{2L}.
\end{align}
\end{thm*}

\begin{proof}
The proof relies on Mirzakhani's integration of geometric functions method. It is close to the proof of the first statement of Theorem~\ref{thm:WX}, see the proof of Lemma 48 in \cite{WX}. We inject our refined expansion of the volumes in the integration formula of Theorem~\ref{thm:integr}.

\begin{align*}
    \mathbb{E}_{WP}(N^{s}_{nsep}(X,L))&=\frac{1}{2}\int_{0}^L x\cdot \frac{V_{g-1,2}(x,x)}{V_{g}}\mathrm{d}x\\
    &=\frac{1}{2}\int_0^Lx\cdot \frac{4}{x^{2}}\cdot \exp\left(\frac{|\mathbf{x}|}{2}-\frac{|\mathbf{x}|^{2}}{32\pi^{2}g}\right)\left(1+o\left(\frac{x^{2}}{g}\right)\right)\\
    &\hspace{6.8cm}\cdot\left(1+\mathcal{O}_{n}\left(\frac{1}{g}\right)\right)\mathrm{d}x\\
    &=\frac{e^{L-\frac{L^{2}}{8\pi^{2}g}}}{2L}\left(1+o\left(\frac{L^{2}}{g}\right)\right).
\end{align*}
The proof of Theorem~\ref{thm:count} is complete. 
\end{proof}

\bibliographystyle{alpha}
\bibliography{bib}

\end{document}